\documentclass[letter, 10pt, conference]{ieeeconf}      

\IEEEoverridecommandlockouts                              
                                                          
\usepackage{cite}

\usepackage{amsmath}
\usepackage{amsfonts}
\usepackage{amssymb}
\usepackage{cuted}
\usepackage{graphicx}
\usepackage{epstopdf}

\usepackage[utf8]{inputenc}
\usepackage{subfig}

\usepackage{supertabular}

\usepackage{savesym}
\savesymbol{AND}
\savesymbol{OR}
\savesymbol{NOT}
\savesymbol{TO}
\savesymbol{COMMENT}
\savesymbol{BODY}
\savesymbol{IF}
\savesymbol{ELSE}
\savesymbol{ELSIF}
\savesymbol{FOR}
\savesymbol{WHILE}
\usepackage{algorithm}
\usepackage{algorithmicx}
\usepackage{algpseudocode}
\usepackage{float}
\floatstyle{boxed}
\newfloat{myprogram}{tb}{alg}
\floatname{myprogram}{Algorithm}

\usepackage{url}

\newcommand{\norm}[1]{\left| {#1} \right|}
\newcommand{\dnorm}[1]{\norm{\norm{{#1}}}}

\newcommand{\ba}{\begin{eqnarray}}
\newcommand{\na}{\end{eqnarray}}
\newcommand{\ban}{\begin{eqnarray*}}
\newcommand{\nan}{\end{eqnarray*}}

\newcommand{\N}{{\mathbb N}}
\newcommand{\R}{{\mathbb R}}

\renewcommand{\a}{\alpha}
\renewcommand{\b}{\beta}

\newcommand{\eps}{\epsilon}
\renewcommand{\d}{\delta}

\newcommand{\p}{\partial}
\renewcommand{\t}{\tau}
\renewcommand{\th}{\theta}
\newcommand{\z}{\zeta}

\newcommand{\ra}{\rightarrow}

\newcommand{\hP}{\hat{P}}
\newcommand{\dhP}{\dot{\hat{P}}}

\newtheorem{theorem}{Theorem}
\newtheorem{lemma}[theorem]{Lemma}
\newtheorem{proposition}[theorem]{Proposition}
\newtheorem{definition}[theorem]{Definition}
\newtheorem{assumption}{Assumption}

\newif\ifcomments
\commentstrue 

\ifcomments
\newcommand\todos[1]{\textcolor{blue}{#1}}
\newcommand{\etc}[1]{~\textit{etc.}}
\newcommand{\etal}[1]{~\textit{et~al.}}
\usepackage[normalem]{ulem}
\else
\newcommand\todos[1]{}
\newcommand{\etc}[1]{}
\newcommand{\etal}[1]{}
\newcommand{\sout}[1]{}
\fi
\usepackage{xcolor}  

\begin{document}
\title{A Geometric Observer for Scene Reconstruction Using Plenoptic Cameras}

\author{Sean G. P. O'Brien,
        Jochen Trumpf,
        Viorela Ila,
        and~Robert Mahony
\thanks{S. O'Brien, J. Trumpf, V. Ila, and R. Mahony are with the Research School of Engineering, Austalian National University, Canberra,
ACT, 2601 Australia, and the Australian Centre for Robotic Vision (ACRV) \protect\url{http://www.roboticvision.org}.
e-mail: \{sean.obrien, jochen.trumpf, viorela.ila, rob.mahony\}@anu.edu.au
}}

\markboth{2018 IEEE 57th Conference on Decision and Control (CDC)}%
{}

\maketitle


\begin{abstract}
This paper proposes an observer for generating depth maps of a scene from a sequence of measurements acquired by a two-plane light-field (plenoptic) camera. 
The observer is based on a gradient-descent methodology. The use of motion allows for estimation of depth maps where the scene contains insufficient texture for static estimation methods to work. A rigourous analysis of stability of the observer error is provided, and the observer is tested in simulation, demonstrating convergence behaviour.
\end{abstract}

\IEEEpeerreviewmaketitle

%

\section{Introduction} \label{sec:Introduction}

\PARstart{D}{epth} estimation is a fundamental problem in computer vision that involves reconstruction of a 3D scene from visual measurements obtained from a camera. This has applications in many areas of engineering including inspection of 3D structures, recreation of 3D scenes for virtual and augmented reality, and  preservation of historical data.

A plenoptic camera is a device which captures light-fields that represent not only the intensity of light over a range of angles at a single point, as with a conventional camera, but the light over a range of points in space.
Plenoptic cameras offer some advantages over other types of cameras for the purpose of depth estimation as depth information is highly correlated with light-field gradients \cite{Dansereau2004CS}. Generating depth maps from plenoptic cameras is an active area of research and there have been numerous developments in recent years \cite{Tao2013ICCV} \cite{Zhang2016CSVT} \cite{Zhang2016CVIU} \cite{Wang2016PAMI}. 
However, research on depth estimation using plenoptic cameras has yet to take into account the additional information provided by temporal correlation of data; existing algorithms consider only frame-by-frame single light-field images.

Recent years have seen new developments in the theory of observers for systems with invariance properties \cite{Lageman2010TAC}, \cite{Aghannan2003TAC} \cite{Bonnabel2009TAC}. This work has led to observer designs based on an internal model principle where observer dynamics consists of an internal model, which tracks the dynamics of the system under observation, combined with an innovation term, which minimises some cost function. The innovation term is typically chosen as the gradient of an error function taken with respect to the state estimate. 

There is an active research community applying an observer design philosophy to computer vision problems \cite{Matthies1989IJCV} \cite{Chen2002TAC} \cite{Chen2004TAC}  \cite{Dahl2007ACC} \cite{Adarve2016RAL} \cite{Keshavan2016CEP} \cite{Grave2015TAC}.  
Early work in this area considered perspective dynamical systems \cite{Chen2002TAC} \cite{Chen2004TAC} \cite{Dahl2007ACC}.  Depth estimation from monocular video data based on extended Kalman filtering (EKF) dates back to the nineties \cite{Matthies1989IJCV}. 
More recent work includes the development of a dynamic filtering algorithm for the computation of dense optical flow in real-time  \cite{Adarve2016RAL}. In \cite{Keshavan2016CEP}, an observer for sparse depth estimation using monocular cameras is proposed. In \cite{Grave2015TAC} an observer for tracking the depth of a given object using from perspective vision data (such as monocular camera data) is formulated that exponentially converges to the object's coordinates. 
 Observers have the advantage that they take into account previously gathered data from the environment rather than only the data available at a given time, allowing for potentially more accurate estimates.
The use of observers also offers some computational advantages. 
Static dense depth estimation methods, which involve minimising a high-dimensional cost function, will typically require tens of update steps per-frame in order to estimate the minimum of the cost function. In contrast, an observer will typically only involve a single cost update step per light-field frame. 

%


In this paper, 
we develop an observer for estimating a dense depth maps of an entire scene provided the camera motion and using light-field measurements as inputs. We follow a general design philosophy for observers by including dynamics of the depth map as an internal model and using the gradient of a disparity map as an innovation term. To the authors understanding, there is no prior work on applying the observer based approach to depth estimation using \textit{plenoptic} camera data.
The use of a moving camera combined with a dynamic observer is found in simulation to relax observability conditions, so that the knowledge of the motion of the light-field camera allows for the estimation of scenes that would not be observable using static depth estimation techniques due to insufficient texture on the scene. 
For such scenarios, points on the estimated scene will remain stationary until such a time when the camera is viewing these points in front of sufficiently textured regions of the scene. In this way we ensure that every point of the estimated scene converges to a point on the actual scene as long as we can guarantee that each point on the scene estimate is viewed at some time in the future.

In Section \ref{sec:Problem_Formulation}, we formulate a mathematical model of light-field cameras, scenes, depth maps, and photometric errors. In Section \ref{sec:observer_derivation}, we formulate the dynamics assigned to point estimates by the observer. We then discuss in Section \ref{sec:Simulation} details of the numerical implementation of the observer, and show its behaviour for a simple simulated scenario.

\section{Problem Formulation} \label{sec:Problem_Formulation}

In this section we develop the geometric framework used to derive the photometric error term minimised by the observer. To this end, we introduce the camera, the scene and the projection models. The physical parameters of the camera model determine the way in which the camera captures light emanating from the scene. This is defined by projections of the scene points to a set of pixels in subimages formed by the camera multi-lenses array. 

\subsection{Light-Field Cameras} \label{sec:Light-Field_Cameras}

\begin{figure}[ht!]
\centering
\includegraphics[width=3.0in]{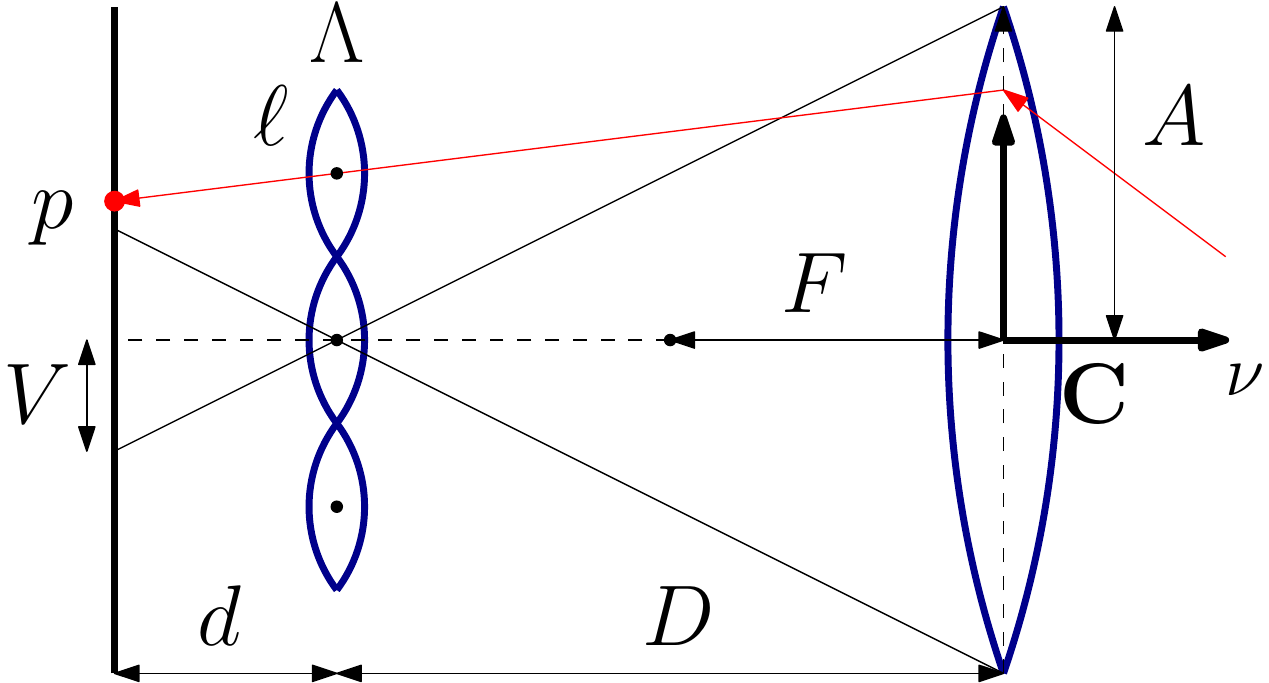}
\caption{The parameters of a focused light-field camera. }
\label{fig:Camera_Model}
\end{figure}

A plenoptic camera is a device that images a scene from a continuous range of points in space, not only from a single optical centre as is the case for monocular cameras. 
Lenslet-based light-field cameras are constructed by positioning a densely-packed array of ``lenslets'' -- lenses of typically microns in diameter across -- between a conventional imaging device, such as a CMOS array, and a focus lens (see Fig. \ref{fig:Camera_Model}). 

We say that a light-field camera maps a \textit{lenslet} $\ell$ and a \textit{pixel} $p$ within the subimage of that lenslet to a colour. The \textit{light-field image} is denoted in this paper as $L: (\ell,p) \mapsto (r,g,b)$, where the latter vector is an RGB colour vector in $[0,1]^3$. 
Each lenslet is positioned somewhere on a plane $\Lambda$ called the pupilar plane.
The pupilar plane is coplanar with two additional objects: the \textit{retinal plane} and the \textit{focal lens}. The distance between the pupilar plane and retinal plane is denoted by $d$ and the distance between the pupilar plane and the focus lens is denoted by $D$ (see Fig. \ref{fig:Camera_Model}). 


As in previous papers on light-field cameras \cite{Dansereau2004CS}, we model the focus lens as a thin-lens. This model has two intrinsic parameters, the \textit{focal length}  $F$ and the \textit{pose} $X$, with $\mathbf{x}$, the position of the \textit{optical centre} of the lens, and $R$, the rotational part of the pose. The pose $X$ is given with respect to a fixed reference frame $\textbf{O}$, and itself defines a body-fixed reference frame $\textbf{C}$ of the entire camera. 
One of the axes of $\textbf{C}$, called the \textit{principal axis}, represents the axis that the focal lens is orthogonal to. We define the camera as facing in the positive $z$-direction in the coordinate system $\textbf{C}$, and call the unit-vector pointing in this direction $\nu$ (see Fig. \ref{fig:Camera_Model}).  



Each lenslet in the pupilar plane $\Lambda$ is modelled as a pinhole camera. A pinhole camera is described by a pose and a distance $d$ of the pinhole from the retinal plane. We assume that all of the lenslets in the camera have the same constant distance $d$ from the common retinal plane. 

We assume that the lenslets all share the same orientation $R$ of the thin-lens, and have a constant distance from the optical centre given by $D$, but that they may be positioned anywhere on the disc $\Lambda$. Hence, their pose is entirely determined by specifying their two coordinates on this plane. In summary, we model the lenslet array as a parametrised set of pinhole cameras positioned on a disc and with a common focal length and orientation. 

Each lenslet projects light coming from the space in front of it onto its retinal plane. Each lenslet shares a common retinal plane with all the other lenslets, but the \textit{subimages} produced by each lenslet do not overlap due to the precise choice of the \textit{aperture} $A$ of the focus lens. As the only light source within the camera comes from the circular focus lens, the subimages produced by each lenslet are also circular, and have a radius $V$ called the \textit{subimage radius} determined by the formula $V = \frac{d}{D} A$ (see Fig \ref{fig:Camera_Model}). A factor $s_p$ relates a pixel as seen in one of its subimages to the physical position of that pixel in space, and is given in metres per pixel.

In summary, a light-field camera is represented by the parameter vector $(X,F,D,d,A, s_p)$.

\subsection{Scenes} \label{subsec:Scenes}

\begin{figure}[ht!]
\centering
\includegraphics[width=3.0in]{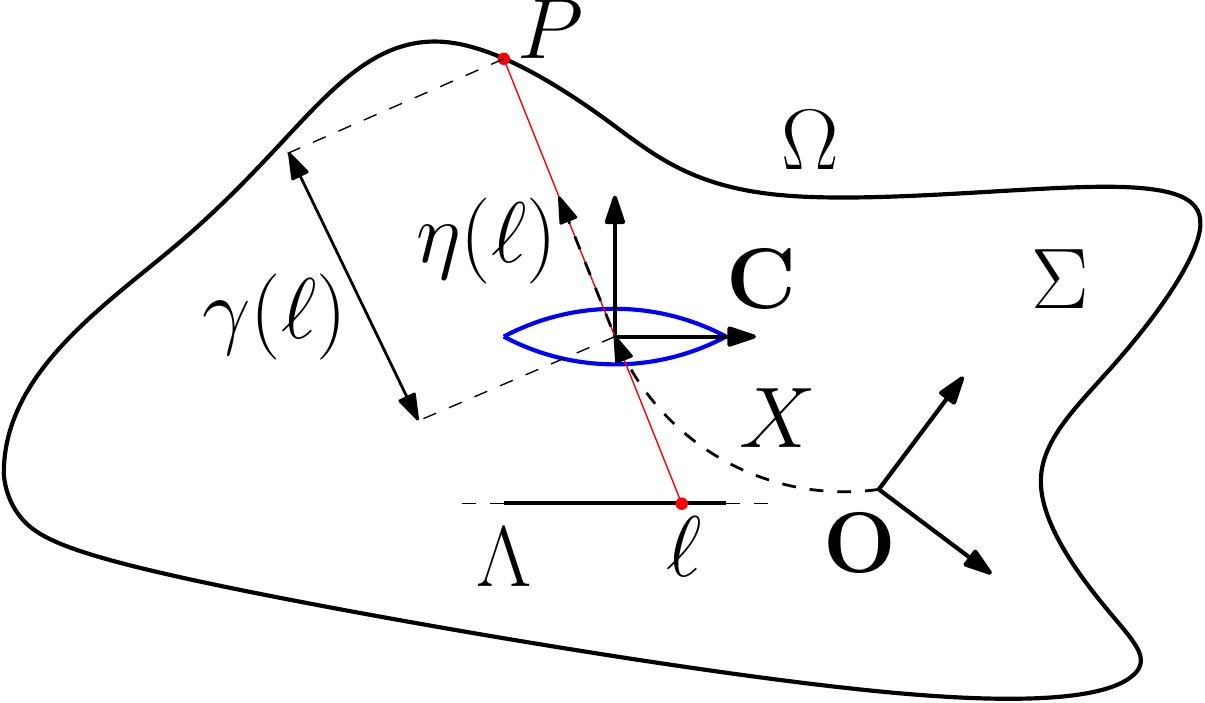}
\caption{A plenoptic camera with pinhole cameras positioned along a plane $\Lambda$ behind a thin-lens shown in blue.}
\label{fig:Scenes_and_Depth_Maps}
\end{figure}

We define an \emph{environment} $\Sigma$ as a piece-wise smooth non-empty open subset of $\R^3$. We define a \emph{scene} $\Omega$ to be the boundary of an environment $\Omega = \p \Sigma$. Defined on a scene is a \emph{brightness map} $\b : \Omega \ra [0, 1]^3$, which assigns an RGB colour vector to each point on the scene $\Omega$. Note that this brightness map does not depend on direction, which in physical terms means we are assuming that the brightness map satisfies a Lambertian condition \cite{Hartley2004}.

Although a scene is a set, we may locally parametrise this set via a \textit{distance map} $\gamma: \Lambda \ra \R_+$. The distance maps used in this paper are defined with respect to the optical centre $\mathbf{x}$ of the camera in the environment $\Sigma$, and take as input a lenslet $\ell$ and return the distance from $\mathbf{x}$ to the point on the scene $\Omega$ in direction \mbox{$\eta(\ell) := (\mathbf{x} - \ell)/ \dnorm{\mathbf{x} - \ell}$}, (see Fig.~\ref{fig:Scenes_and_Depth_Maps}). Together, a brightness map $\b$ and a distance map $\gamma$ are sufficient to represent the visible portion of a scene. 

\subsection{Projection Model} \label{subsec:Projection_Model}

\begin{figure}[ht!]
\centering
\includegraphics[width=3.0in]{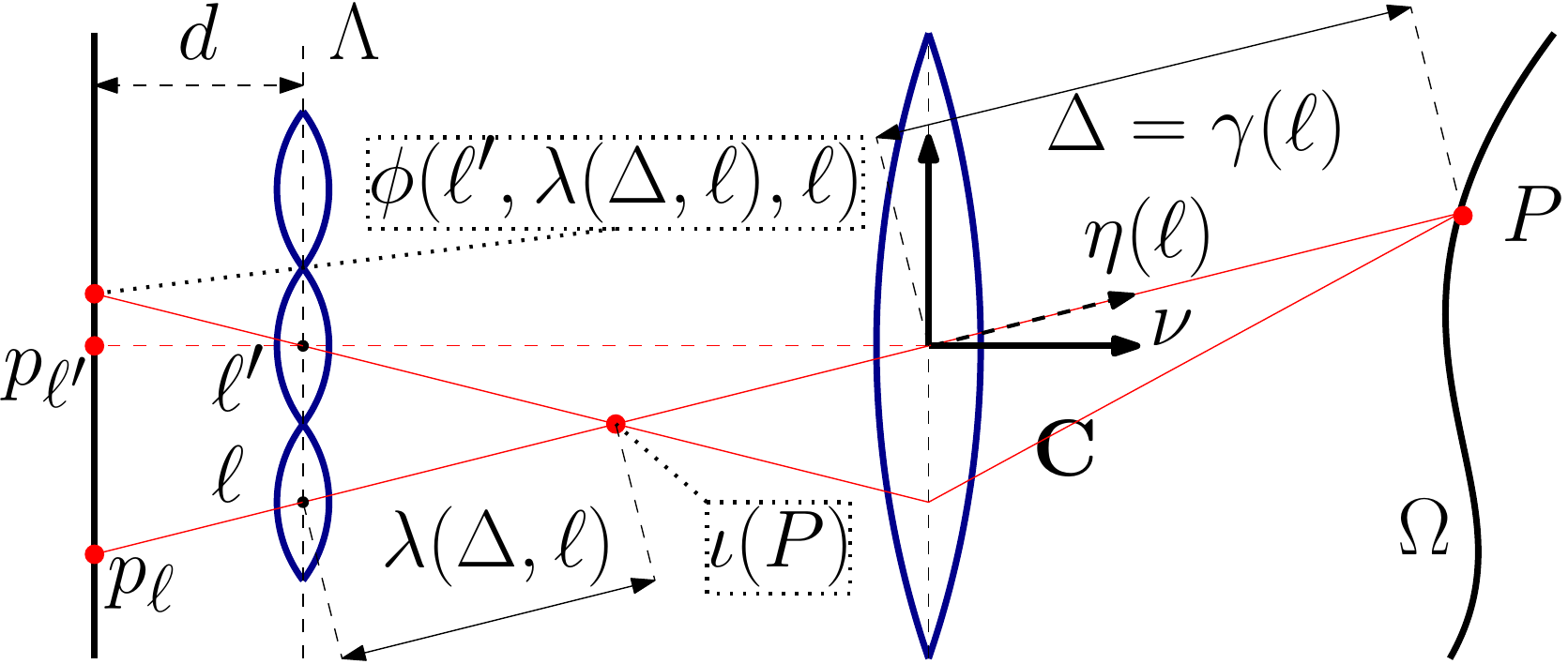}
\caption{The thin-lens projection of the point $P$ on the scene $\Omega$ is given by $\iota(P)$. The point $P$ has a distance of $\Delta = \gamma(\ell)$ from the optical centre of the camera in direction $\eta(\ell)$, corresponding to a distance $\lambda(\Delta,\ell)$ of its corresponding image point $\iota(P)$ from the lenslet $\ell$. The perspective projection $\phi(\ell',\lambda(\Delta,\ell),\ell)$ of \mbox{$\iota(P) = \ell + \lambda(\Delta,\ell) \eta(\ell)$} through lenslet $\ell'$ is illustrated.}
\label{fig:Image_Points_Projections}
\end{figure}


We assume that the rays of light are being emitted from a scene, and that the colour of these rays are determined entirely by which point on the scene they are emitted from, therefore we can treat the light-field camera as projecting each point on a scene to a set of points on the retinal plane. 

We therefore model how the light-field camera records light in two steps: first, for a given point $P$ on the scene $\Omega$, we project the point $P$ through the thin-lens, resulting in an \textit{image point} $\iota(P)$, then we project this image point through each of the lenslets in $\Lambda$ for which the image point is visible via the perspective projections associated with those lenslets. 

The projection through the focus lens is determined by the well-known thin-lens equation \cite{Hartley2004}. 
Points in the environment, unless otherwise specified will be assumed to have coordinates expressed in the body-fixed frame $\textbf{C}$. We will represent these points in terms of the image of a depth map, so that $P = \gamma(\ell) \eta(\ell)$. The image point $\iota(P)$ corresponding to the point $P$ is then given by \cite{Hartley2004}
\begin{equation} \label{eq:Image_Map}
\iota(P) = \frac{F}{F - P \cdot \nu} P, 
\end{equation}
where $F$ is the focal length and $\nu$ is the direction the camera is facing in, see Fig. \ref{fig:Image_Points_Projections}. The image point $\iota(P)$ has the special property that any ray of light passing through it also passes through $P$ after this ray is refracted by the thin-lens. 


Similarly to the distance map $\gamma$, it is convenient to define a ``virtual'' distance map $\lambda: \R^+ \times \Lambda \ra \R$ which defines the ``virtual scene'' $\iota(\Omega)$. 
The algebra describing the perspective projection through each lenslet $\ell$ is simplified by expressing the distance of an image point $\iota(P)$ as its distance to the lenslet. Because of this, we define the virtual distance \mbox{$\delta = \lambda(\Delta,\ell)$}, corresponding to a real distance $\Delta = \gamma(\ell)$ where $\ell \in \Lambda$ is a lenslet, to be the distance of the point on the virtual scene $\iota(\Omega)$ from the lenslet $\ell$ in direction $\eta(\ell)$. Note that $\lambda(\Delta,\ell)$ can be negative, unlike the real distance $\Delta$. 
With that, the virtual distance $\delta$ corresponding to distance $\Delta$ is given by
\begin{equation} \label{eq:Virtual_Depth}
\delta = \lambda(\Delta, \ell) = \frac{F (\Delta \eta(\ell) \cdot \nu))}{F - (\Delta \eta(\ell) \cdot \nu)} - \ell \cdot \eta(\ell).
\end{equation}


Given a point $Q \in \R^3$ and a specified plane $\Gamma$, we define for any point $P \in \R^3$ that satisfies $P \cdot x < 0$ for all $x \in \Gamma$ (meaning that $Q$ is between $P$ and $\Gamma$), the projection $\pi^{\Gamma}_{Q}(P)$ as the point of intersection of the line passing through both $Q$ and $P$ with $\Gamma$. We omit $\Gamma$ from the notation whenever the meaning is clear from context. 

 We define the map $\phi$ as $\phi(\ell',\delta,\ell) := \pi_{\ell'}(\ell + \delta \eta(\ell))$ where the plane $\Gamma$ is taken to be the retinal plane of the lenslet $\ell'$. 
The map $\phi$ is derived via a similar triangles argument and is explicitly given by
\begin{equation} \label{eq:Lenslet_Projection}
\phi(\ell',\delta,\ell) = \frac{d}{\delta \eta(\ell) \cdot \nu} (\ell' - \ell - \delta \eta(\ell)) + \ell,
\end{equation}
where $d$ is the lenslet focal length, see Fig. \ref{fig:Image_Points_Projections}. Given that each lenslet has the same limited subimage radius $V$ (cf. Section \ref{sec:Light-Field_Cameras}), not all lenslets will have a given image point $\iota(P)$ visible in their subimages. An image point $\iota(P)$ will only be visible to a given lenslet $\ell'$ if the perspective projection of $\iota(P)$ through $\ell'$ is within distance $V$ of the central pixel $p_{\ell'}$ of the lenslet $\ell'$, see Fig. \ref{fig:Camera_Model} and Fig. \ref{fig:Image_Points_Projections}.

The set $W(\Delta,\ell)$ is the set of lenslets $\ell'$ in $\Lambda$ for which $\dnorm{\phi(\ell',\lambda(\Delta,\ell),\ell)-p_{\ell'}} < V$, i.e. the set of lenslets for which the image point $\iota(P) = \ell + \lambda(\Delta,\ell) \eta(\ell)$ is visible.  


In summary, a point $P$ in the environment is projected to a lenslet image point in a two-step process given by eq. \eqref{eq:Image_Map} and \eqref{eq:Lenslet_Projection}, where $\ell' \in W(\Delta,\ell)$. 

\subsection{Photometric Errors Associated With Distance Maps} \label{subsec:photometric_errors}

\begin{figure}[ht!]
\centering
\includegraphics[width=3.0in]{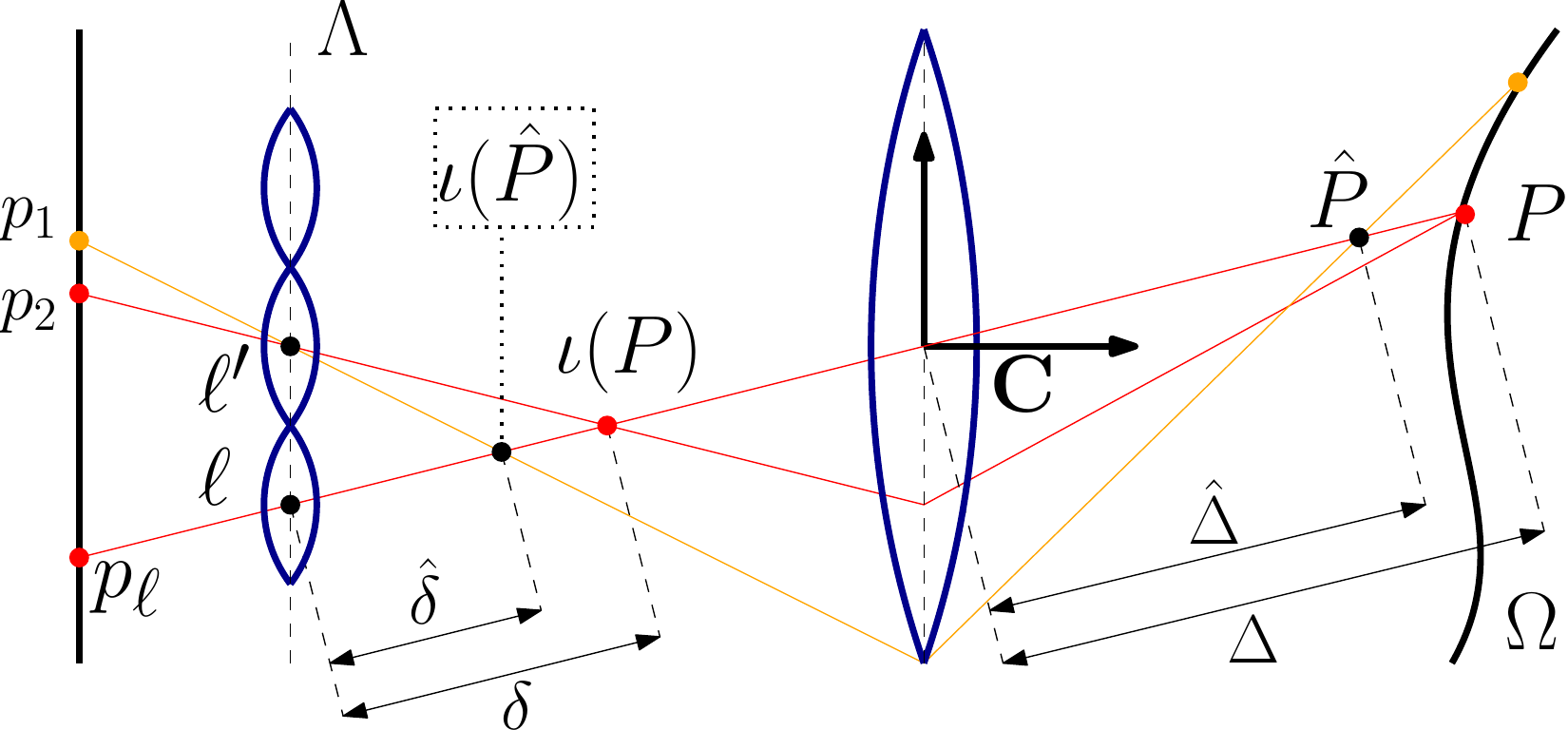}
\caption{A true distance $\Delta$ is shown together with an incorrect distance estimate $\hat{\Delta}$. These distances correspond to virtual distances $\delta$ and $\hat{\delta}$, respectively. 
The ray with coordinates $(\ell',p_2)$ has the same colour as the ray with coordinates $(\ell,p_\ell)$, but the ray with coordinates $(\ell',p_1)$ does not.}
\label{fig:Photometric_Error}
\end{figure}

Now, we have developed the framework necessary to state the photometric error which will be minimised by the observer. Suppose that the camera is positioned somewhere in the environment $\Sigma$ with pose $X$, that the true distance of the scene in direction $\eta(\ell)$ is $\Delta = \gamma(\ell)$, and that we have a distance estimate $\hat{\Delta}$ and the light-field image $L$. 

The ray of light which passes through both the lenslet $\ell$ and the point $\Delta \eta(\ell)$ is the same ray of light which passes through $\ell$ and $\hat{\Delta} \eta(\ell)$ for any distance estimate $\hat{\Delta}$. Therefore, if the distance estimate $\hat{\Delta}$ is accurate, we should expect that all other rays passing through the point $\hat{\Delta} \eta(\ell)$ have the same colour, assuming a Lambertian constraint on the colour map $\beta$, see Fig. \ref{fig:Photometric_Error}.

Therefore, the sum of absolute differences between the colours of all other rays passing through $\hat{\Delta} \eta(\ell)$ and the \textit{central ray} associated with $\ell$ -- that is the ray passing through both $\ell$ and the optical centre of the camera -- should be minimised by accurate distance estimates. 

We define the square of the absolute difference in colour between a central ray of a lenslet $\ell \in \Lambda$ and a ray passing through both another lenslet $\ell' \in \Lambda$ and a point estimate $\hat{\Delta} \eta(\ell)$ as the following \emph{pairwise lenslet error} function $e$ $$e(\ell',\hat{\Delta},\ell) := \dnorm{L(\ell,p_\ell) - L(\ell',\phi(\ell',\lambda(\hat{\Delta},\ell),\ell))}^2.$$

Because in practice, a plenoptic camera only has lenslets positioned on a subset $\Lambda^* \subset \Lambda$ that is non-empty, bounded,  convex and open relative to $\Lambda$, we will only update depths assigned to lenslets $\ell$ on this set $\Lambda^*$. However, we will assume that we have light-field information available to us outside of this set in order to ensure differentiability properties of the error function. In practice, this means that for any bounded, convex and relatively open subset of lenslets there is a maximum distance for which we can ensure the local error function defined below is  continuously differentiable.   

Let $\hat{Q}^z = \iota(\hat{\Delta} \eta(\ell)) \cdot \nu$, be the $z$-component of the image of a point estimate $\hat{P}$ of distance $\hat{\Delta}$ corresponding to a lenslet $\ell$. We propose that given a lenslet $\ell$, a distance estimate $\hat{\Delta}$, and a light-field image $L$, the following local error function $\eps$ should be minimised by accurate estimates of the distance: 
\begin{equation} \label{eq:photometric_error}
\eps(\hat{\Delta},\ell) := \left( 1 + \frac{D}{\hat{Q}^z} \right)^{-2} \int_{W(\hat{\Delta},\ell)} e(\ell',\hat{\Delta},\ell) d \ell'.
\end{equation}
\begin{figure}[t]
\centering
\includegraphics[width=3.0in]{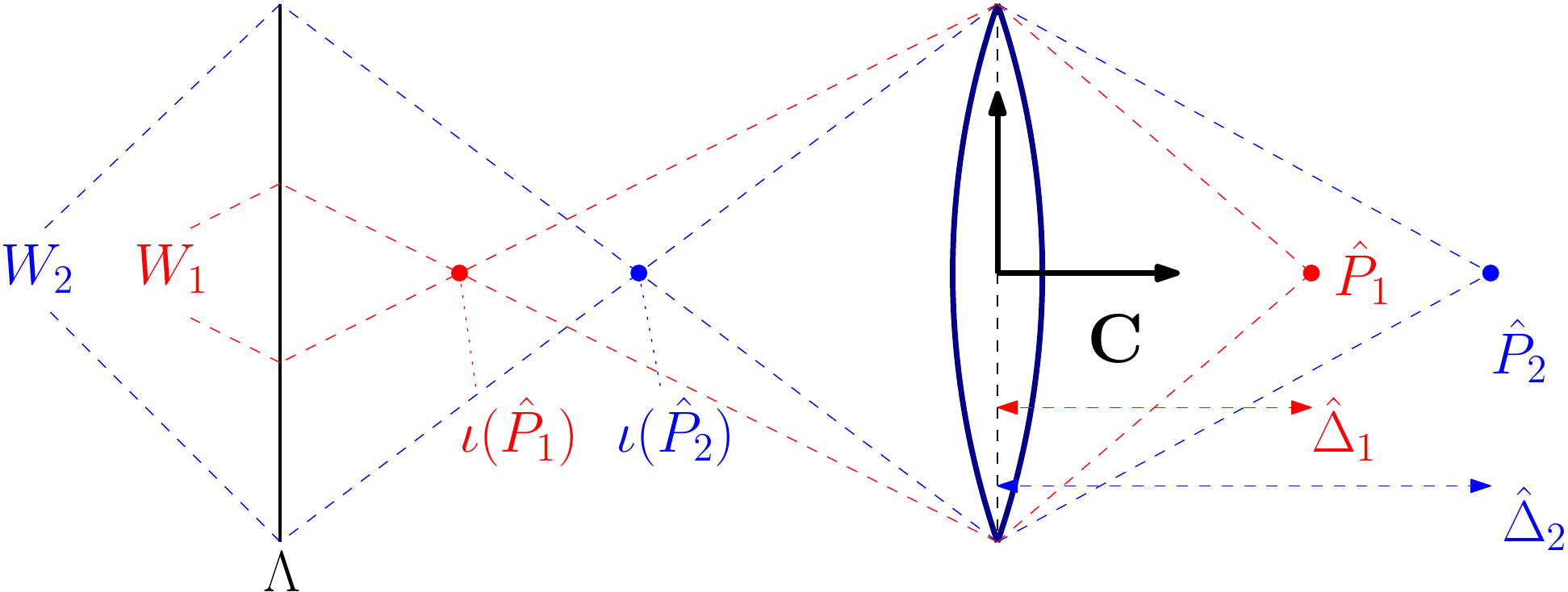}
\caption{The windows $W_1 = W(\hat{\Delta}_1,\ell)$ and $W_2 = W(\hat{\Delta}_2,\ell)$ corresponding to depth estimates $\hat{\Delta}_1$ and $\hat{\Delta}_2$ where $\hat{\Delta}_1 < \hat{\Delta}_2$.}
\label{fig:Window_Size}
\end{figure}
The purpose of the factor before the integral is to counteract the effect of the varying size of the window $W(\hat{\Delta},\ell)$ which will otherwise result in smaller errors for smaller distance estimates, regardless of the correctness of these estimates, See Fig. \ref{fig:Window_Size}. 

It is the gradient of this error function with respect to estimated depth which will be used to update point estimates. 

\section{Observer Derivation} \label{sec:observer_derivation}

In this section, we use the error function $\eps(\hat{\Delta},\ell)$ defined in the previous section to derive an observer based on the gradient of this error map. Note that other error functions could be considered, but are beyond the scope of this paper. The trajectories of point estimates given by this observer are shown in the appendix to have limit points on the scene $\Omega$, given some assumptions on the scene $\Omega$, brightness map $\b$, and camera trajectory $X_t$. 

Because the scene is stationary in reference frame $\mathbf{O}$, it is easiest to express the dynamics of point estimates in this reference frame, as it makes the internal model term trivial, since for points $P$ on the scene $\dot{P}(t) = 0$ in frame $\mathbf{O}$. Therefore, the internal model term in the observer will also be trivial for all point estimates. 

Because we are now expressing the various maps used in this derivation in frame $\mathbf{O}$, we index several of the functions and variables which are dependent on time by $t$. These include the camera's pose $X_t$ expressed in $\textbf{O}$, the pupilar plane $\Lambda_t$ and subset $\Lambda^*_t$ as subsets expressed in $\mathbf{O}$, the camera's optical centre $\mathbf{x}_t$, the direction map $\eta_t$, and the light-field $L_t$. 

For a given point $\hP \in \R^3$ expressed in the fixed coordinate frame $\mathbf{O}$, let $\ell_t = \pi_{\mathbf{x}_t}(\hP)$, then we define
\begin{equation} \label{eq:observer_vector_field}
v_t(\hP) := \begin{cases}
0-\nabla_1 \eps(\hP \cdot \eta_t(\ell_t),\ell_t) \eta(\ell_t), & \ell_t \in \Lambda^*_t, \\
0, & \text{otherwise}.
\end{cases}
\end{equation}

The innovation term here is the gradient of the error function $\eps$ with respect to the first argument. The first argument of this error function is the distance of a point from the optical centre of the camera. Hence, the innovation term serves to update distances of point estimates in the direction that minimises the photometric error associated with the point estimate. Because we assume that the actual scene is stationary, the internal model for each point is $0$ as they are not moving in reference frame $\mathbf{O}$. 

The observer updates a point estimate with starting position $\hP_0$ according to
the time-varying vector field $v_t$, so that
\begin{equation} \label{eq:Observer_Estimate}
\dot{\hP}_t := v_t(\hP_t).
\end{equation}


The piecewise definition of $v_t$ reflects the fact that we are only updating depths for lenslets $\ell \in \Lambda^*_t$. 
The method is shown to provide accurate point estimates in simulation, see Section \ref{sec:Simulation}. 
A proof of convergence of solutions of \eqref{eq:Observer_Estimate} to the true values is given in Appendix \ref{sec:proof}.

\section{Theoretical Results} \label{sec:Theory}

Convergence of the observer design given in Section  \ref{sec:observer_derivation} is stated by Theorem \ref{thmE1} which states that a point estimate defined by \eqref{eq:Observer_Estimate} converges to the actual scene $\Omega$. This result holds given the following assumptions. However, this does not mean that the listed assumptions are the weakest possible to ensure asymptotic convergence. 

In order to avoid unnecessary discussions of the subtleties of solution concepts for differential equations with discontinuous right hand side \cite{filippov1988},
we assume existence and uniqueness of absolutely continuous solutions of \eqref{eq:Observer_Estimate} for all initial conditions. This will be the case for reasonable camera trajectories.

We denote the topological closure of a set $S\in\R^3$ by $\text{cl}(S)$.

\begin{definition}
The set $C^+(B,\hP)$ is the positive half-cone with apex $\hP \in \R^3$ spanned by the bounded convex set $B \subset \R^3$, where $\hP \not\in B$, see Fig. \ref{fig:Sub-cones}. 
Formally, it is the set of $\hP' \in \R^3$ for which there exists a point $\mathbf{x}' \in B$ and an $\a > 0$ such that $\hP' - \hP = \a (\hP - \mathbf{x}')$.
The set $C^+(B,\hP)$ is open whenever $B$ is, does not contain the apex $\hP$, and extends to infinity. 
We denote $C^+_0(B,\hP)=C^+(B,\hP)\cup\{\hP\}$.
The negative half-cone, $C^-(B,\hP)$, is defined as the set of $\hP' \in \R^3\setminus\text{cl}(B)$ for which there exists a point $\mathbf{x}' \in B$ and an $0<\a<1$ such that $\hP' - \hP = -\a (\hP - \mathbf{x}')$. 
The set $C^-(B,\hP)$ is open whenever $B$ is, does not contain the apex $\hP$, is bounded and sits atop the base $B$. 
\end{definition}

The following constant defines the minimum depth a point has if the image of that point lies between the focal lens and the pupilar plane: 
\begin{align*}
\Delta_{\text{min}}:=\frac{1}{\inf_{\ell \in \Lambda^*} (\eta(\ell) \cdot \nu)} \max \left(F,\frac{D F}{F - D} \right).
\end{align*}

\begin{assumption} \label{assume:camera_contained}
$X_t$ is continuous in $t$ and there exists an open ball $\mathbf{B} \subset \Sigma$ centred at $0$ in reference frame $\mathbf{O}$ such that both the optical centre $\mathbf{x}_t$ and the bounded cone $\{ Q \in C^+(\Lambda^*_t,\mathbf{x}_t)\,|\, Q \cdot \nu_t \leq \Delta_{\text{min}}\}$ are contained within $\mathbf{B}$ for all $t \geq 0$.   
\end{assumption}

This assumption ensures that the camera moves in a continuous fashion and never gets too close to the scene, allowing us to pick initial conditions of at least distance $\Delta_\text{min}$ away from the focal lens of the camera. 

\begin{assumption} \label{assume:increasing_colour_diff}
Let $P$, $x_1$, $x_2 \in \Omega$. If $\dnorm{x_1 - P} > \dnorm{x_2 - P}$ then $\dnorm{\b(x_1) - \b(P)} > \dnorm{\b(x_2) - \b(P)}$.
\end{assumption}

This assumption states that the colour map is monotonic. This is one assumption which may potentially be weakened in future work.

\begin{assumption} \label{assume:convexity}
The scene $\Omega$ is a convex surface.
\end{assumption}

This assumption may be weakened in future work to the scene being star-shaped with respect to $\textbf{B}$ from Assumption~\ref{assume:camera_contained}. 

It will also be convenient in the following proof to define the set of times for which a given point estimate $\hP_t$ is seen by the camera.  

\begin{definition}\label{def:TP0}
Given an initial condition $\hP_0$ of the system \eqref{eq:Observer_Estimate}, define $T(\hP_0)$ to be the set of times $t > 0$ for which $\pi_{\mathbf{x}_t}(\hP_t) \in \Lambda^*_t$ and $\hP_t \cdot \nu_t > 0$. 
\end{definition}

Note that $t \in T(\hP_0)$ implies that $\hP_t \in C^+(\Lambda^*_t,\mathbf{x}_t)$ and $\dhP_t = -\nabla_1 \eps(\hP_t \cdot \eta_t(\ell_t),\ell_t) \eta_t(\ell_t),$ where $\ell_t = \pi_{\mathbf{x}_t}(\hP_t) \in \Lambda^*_t$.

Lastly, we wish to ensure that there is always a future interval of time for which a given point estimate, and a neighbourhood around it, will be seen by the camera. 
Let $B_r(P) \subset \R^3$ denote the open ball of radius $r > 0$ centred at $P \in \R^3$.

\begin{assumption} \label{assume:persistence}
There exists a $\rho > 0$ and a $\Delta t > 0$ such that for a given initial condition $\hP_0$, and all times $t > 0$ there exists a $t^+ > t$ such that $\pi_{\mathbf{x}_s}(\text{cl}(B_{\rho}(\hP_s))) \subset \Lambda^*_s$ and $\hP' \cdot \nu_s > 0$ for all $\hP' \in \text{cl}(B_{\rho}(\hP_s))$ and for all $s \in [t^+, t^+ + \Delta t]$.  
In particular, $[t^+, t^+ + \Delta t]\subset T(\hP_0)$.
\end{assumption}

Under these assumptions, we have the following result. 

\begin{theorem} \label{thmE1}
Let $\hP_0\in\Sigma$ and $\hP_0\not\in\text{cl}(\mathbf{B})$, where $\mathbf{B}$ is from Assumption \ref{assume:camera_contained}. Then there exists a point $P \in \Omega$ such that $\lim_{t \ra \infty} \hP_t = P$.
\end{theorem}

\section{Simulation} \label{sec:Simulation}

The observer derived in the previous section was verified in simultation for a simple scenario. In order to do this, synthetic light-field data was generated. In our simulations, light-field data was represented by a large $m \times M$ by $n \times N$ resolution image where $m \times n$ is the resolution of the subimage produced by a single lenslet $\ell$, and $M \times N$ is the number of lenslets. 

The light field camera is modelled as a rectangular array of lenslets positioned in front of a rectangular array of pixels.  The colour assigned to a pixel $p$ in the subimage of lenslet $\ell$ is generated using ray-tracing. 
The pixel location is where the ray passing through $p$ and $\ell$ is refracted to and can be calculated using \eqref{eq:Image_Map}. The colour assigned to the lenslet-pixel pair $(\ell,p)$ is then given by the colour $\beta(P)$ of the point $P$ on the 3D scene where the refracted ray corresponding to $(\ell,p)$ intersects the scene.

In the current implementation, the scene estimates are represented using a point-cloud. Since we are only using a discrete number of lenslets and pixels, an appropriate discretisation of the point-estimate update in \eqref{eq:Observer_Estimate} must be calculated. The choice used in this paper is as follows. For a given point-estimate $\hat{P}_t$ at time $t$, the perspective projection $\pi_{\mathbf{x}_t}(\hat{P}_t)$ of the point-estimate onto the plane of distance $D$ behind the optical centre $\mathbf{x}_t$ is first calculated. We then determine whether $\pi_{\mathbf{x}_t}(\hat{P}_t)$ lies in $\Lambda^*_t$. If not, it is assigned $0$ velocity. Otherwise, if the projection is found to lie within the bounds of $\Lambda^*_t$, we find the nearest lenslet $\ell$ to $\pi_{\mathbf{x}_t}(\hat{P}_t)$ and assign to $\hat{P}_t$ the velocity $- \nabla_1 \eps(\hat{P}_t \cdot \eta_t(\ell), \ell) \eta_t(\ell)$ in accordance with \eqref{eq:observer_vector_field}. Once all velocities have been assigned to all points, we update the point estimates with these velocities using some positive gain $K$.

\subsection{Results}

\begin{figure}[h!] 
\begin{tabular}{ccc} 
\subfloat[Actual Scene]{\includegraphics[width = 4cm]{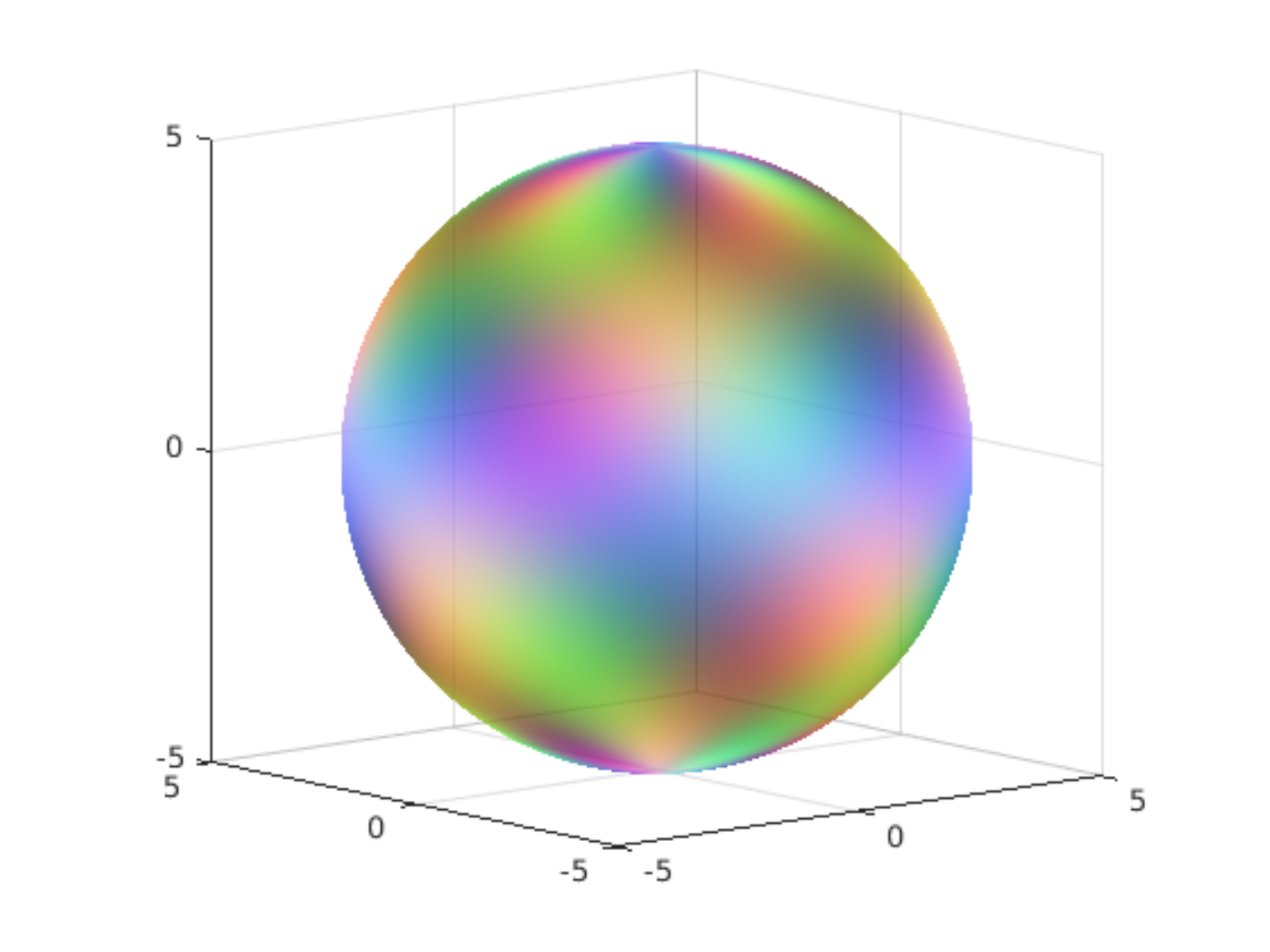}} &
\subfloat[Final Scene Estimate]{\includegraphics[width = 4cm]{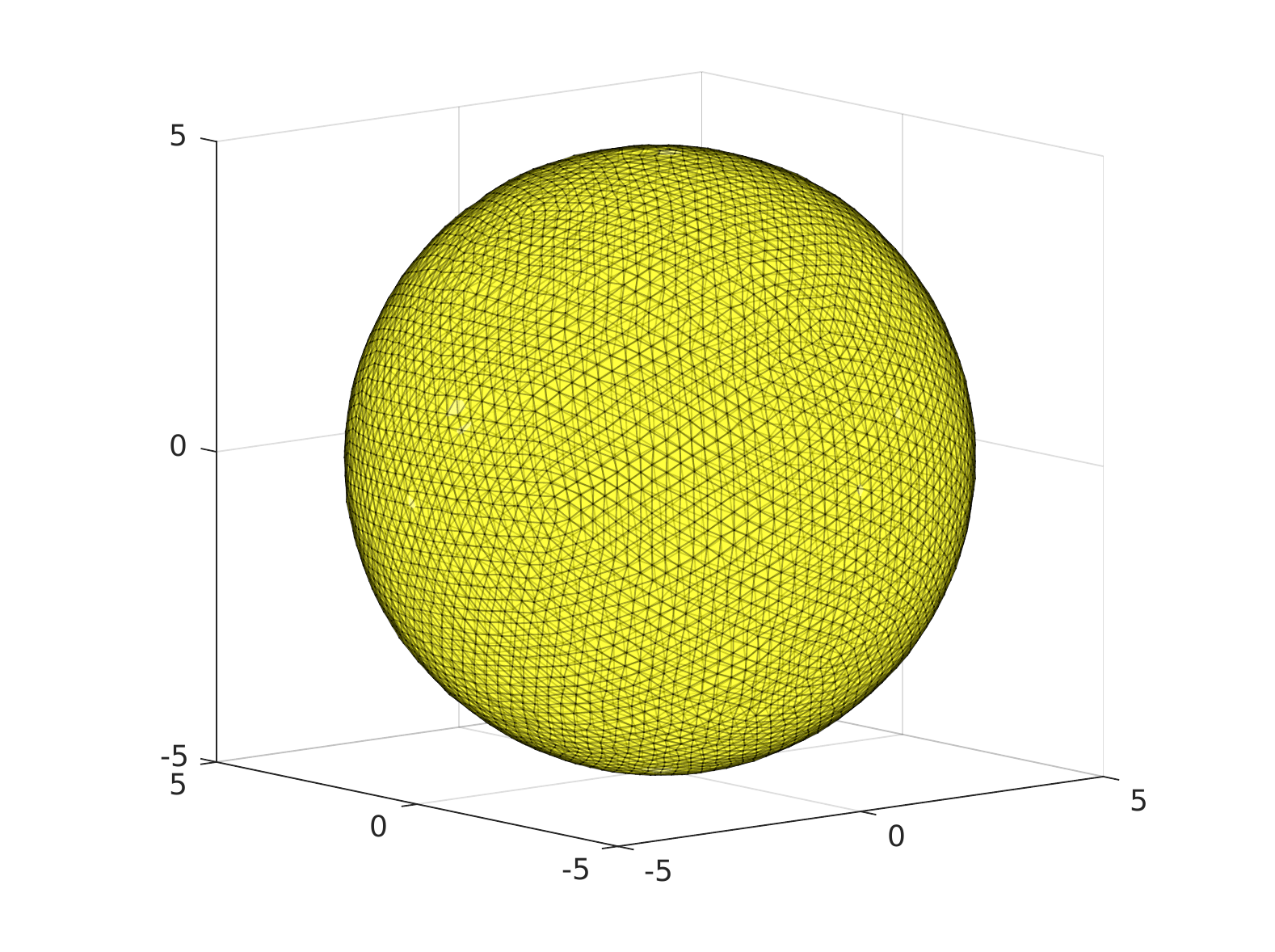}} 
\end{tabular}
\caption{Actual scene with colour map (left), and final scene estimate at frame 5000 (right).} \label{fig:Scene_Comparison}
\end{figure}

\begin{figure}[ht!] 
\includegraphics[width = 8.5cm,trim={4cm 12.8cm 4cm 12.8cm},clip]{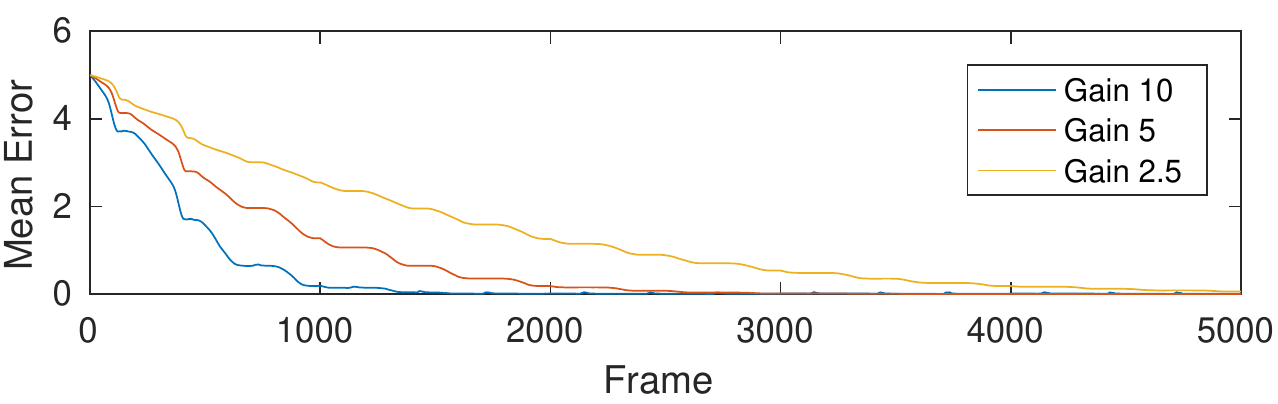}
\caption{Transient response of the average distance of each point estimate from the scene for various gains up to frame 5000.} \label{fig:Error_Graph}
\end{figure}

In this simulation, the scene is a sphere and colour was assigned to every point on its surface based as a function of its Euclidean coordinates in $\mathbf{O}$. 

The camera followed a path determined by a Lissajous figure and was made to always face outwards from the sphere. This path ensured that each point on the scene is viewed from slightly different perspectives multiple times, which assists with minimising the accumulation of numerical error which may occur from using the same frame multiple times. A practical application that allows essentially free design of camera trajectories is 3D scanning of environments for the purpose of  map or model building. In the following simulation, the camera follows such a trajectory lasting 5000 frames.  

The initial scene estimate is given by a surface generated from subdividing the faces of an icosahedron \cite{Wright2014a}. 
The total error graph in Fig. \ref{fig:Error_Graph} shows that with a well chosen gain the observer converges to the scene with a small steady-state error after around 2000 frames, which corresponds to 10--20 iterative updates of each point of the scene. The total error of a scene estimate is given here as the sum of the squares of the distances of each vertex on the scene estimate to the actual scene. 

Since the field of view of the camera is small compared to the total area of the scene, a large number of frames are required in order to ensure convergence of the entire scene. A comparison of the scene shown side-by-side with the real scene is given in Fig. \ref{fig:Scene_Comparison}. 

%

Choice of gain and camera trajectories were seen to be important factors 
when running the proposed algorithm on more challenging scenes.
Too large a gain can result in overshoot, causing point estimates to oscillate or diverge, whereas too small a gain results in very slow convergence. 
A necessary condition for practical convergence of each point estimate to the scene appears to be that each point on the scene is repeatedly updated and repeatedly viewed from different perspectives, including perspectives that increase the visual contrast in a neighbourhood of the point. The first part of this statement is also corroborated by the conditions needed for the convergence proof in Appendix \ref{sec:proof}, cf. Assumption \ref{assume:persistence}.

\section{Conclusion} \label{Conclusion}
In this paper, we develop an observer that uses known camera trajectories and light-field measurements to produce estimates of depth maps. This observer design is based on the internal model principle. The proposed observer exploits the concept of plenoptic cameras as continuous sets of pinhole cameras to derive an innovation term given by the gradient of an integral error. The asymptotic convergence of the observer error to zero is proven for scenes satisfying some basic assumptions. The correctness of the observer algorithm is illustrated using a simulation of a simple scene. 
Future work includes experimentation with different, more robust error functions and experimentation with actual light-field video camera data. 

\appendices
\section{Proof of Convergence} \label{sec:proof}

\subsection{Cone Geometry} \label{subsec:cone_geometry}

\begin{figure}[ht!]
\centering
\includegraphics[width=3.0in]{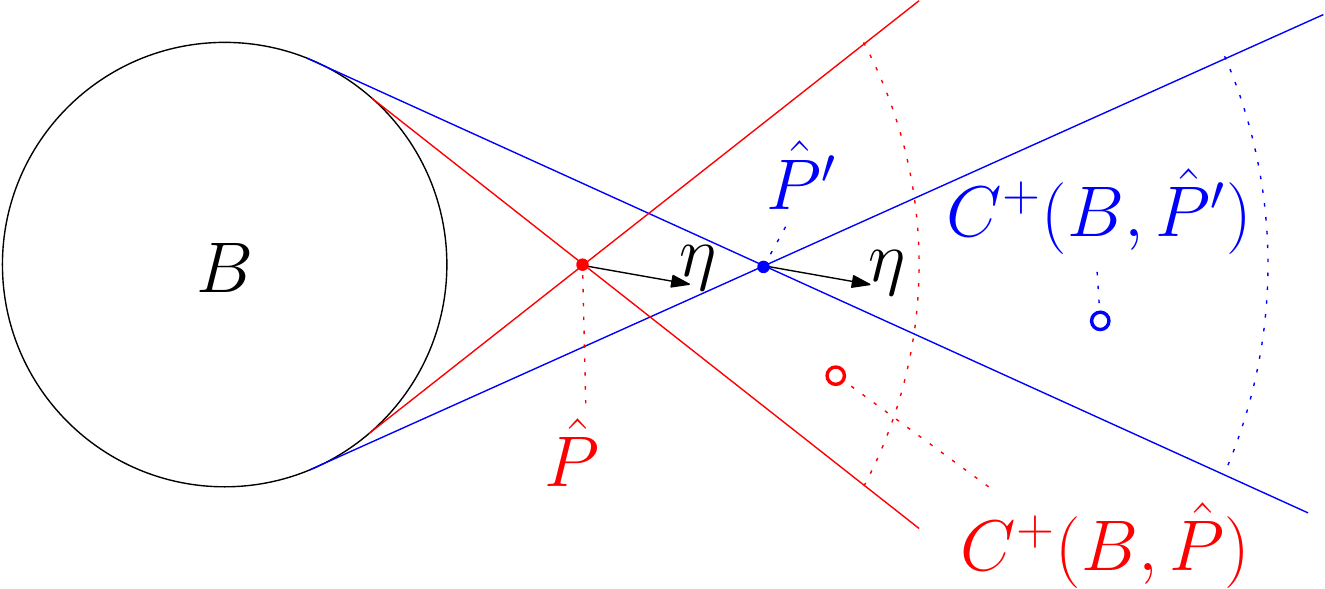}
\caption{A Planar cut through $B$ that contains both $\hat{P}$ and $\hat{P}'$.}
\label{fig:Sub-cones}
\end{figure}

\begin{proposition} \label{prop:cone_reversal}
Let $B$ be a bounded convex set and $\hP\not\in\text{cl}(B)$. Then $\hP' \in C^+(B,\hP)$ if and only if $\hP \in C^-(B,\hP')$.
\end{proposition}
\begin{proof}
If $\hP' \in C^+(B,\hP)$ then there exists an $\mathbf{x}' \in B$ and an $\a > 0$ such that $\hP'-\hP = \a(\hP-\mathbf{x}')$ which implies $\hP-\hP' = \frac{-\a}{1+\a} (\hP'-\mathbf{x}')$ and hence $\hP \in C^-(B,\hP')$.
Conversely, if $\hP \in C^-(B,\hP')$ then there exists an $\mathbf{x}' \in B$ and an $0<\a<1$ such that $\hP-\hP' = -\a(\hP'-\mathbf{x}')$ which implies $\hP'-\hP = \frac{\a}{1-\a} (\hP-\mathbf{x}')$ and hence $\hP' \in C^+(B,\hP)$.
\end{proof}

\begin{proposition} \label{prop:cones_generated_from_elements_are_subcones}
Let $B$ be an open ball. 
If $\hP' \in C^+(B,\hP)$ then $\text{cl} (C^+(B,\hP')) \subset C^+(B,\hP)$. 
Furthermore, if $\hP'\in C^+(B,\hP)$ and $\hP' + \eta \in C^+(B,\hP')$ then $\hP + \eta \in C^+(B,\hP)$.
If $\hP' \in C_0^+(B,\hP)$ then $C^+(B,\hP') \subset C^+(B,\hP)$.
\end{proposition}
\begin{proof}[Sketch of Proof]
Picture a planar cut through $B$ that contains both $\hP$ and $\hP'$ (see Fig. \ref{fig:Sub-cones}) and note that $C^+(B,\hP)$ is on the opposite site of $\hP$ to $B$. Since $\hP'$ is inside the open cone $C^+(B,\hP)$, the opening angles of $C^+(B,\hP')$ are strictly smaller than those of $C^+(B,\hP)$ and the first result follows. 
Translating the cone $C^+(B,\hP')$ to $C^+(B,\hP') - \hP' + \hP$ results in a cone with apex $\hP$ which has smaller opening angles than $C^+(B,\hP)$ and is therefore a subset of it, giving the second result.
The third result follows from the first observing that $C^+(B,\hP') = C^+(B,\hP)$ if $\hP'=\hP$. 
\end{proof}

\begin{proposition} \label{prop:negative_subcone_lemma}
Let $B$ be an open ball. If $\hP \in C^-(B,\hP')$ then $C^-(B,\hP) \subset C^-(B,\hP')$.
\end{proposition}
\begin{proof}[Sketch of Proof]
Picture a planar cut through $B$ that contains both $\hP$ and $\hP'$ (see Fig. \ref{fig:Sub-cones}) and note that both $C^-(B,\hP)$ and $C^-(B,\hP')$ are bounded by the spherical base $B$. Since $\hP$ is inside the open cone $\hP \in C^-(B,\hP')$, the opening angles of $C^-(B,\hP)$ are strictly larger than those of $C^-(B,\hP')$ and hence the cone $C^-(B,\hP)$ touches the spherical base inside $C^-(B,\hP')$. The result follows.
\end{proof}

\begin{figure}[ht!]
\centering
\includegraphics[width=3.0in]{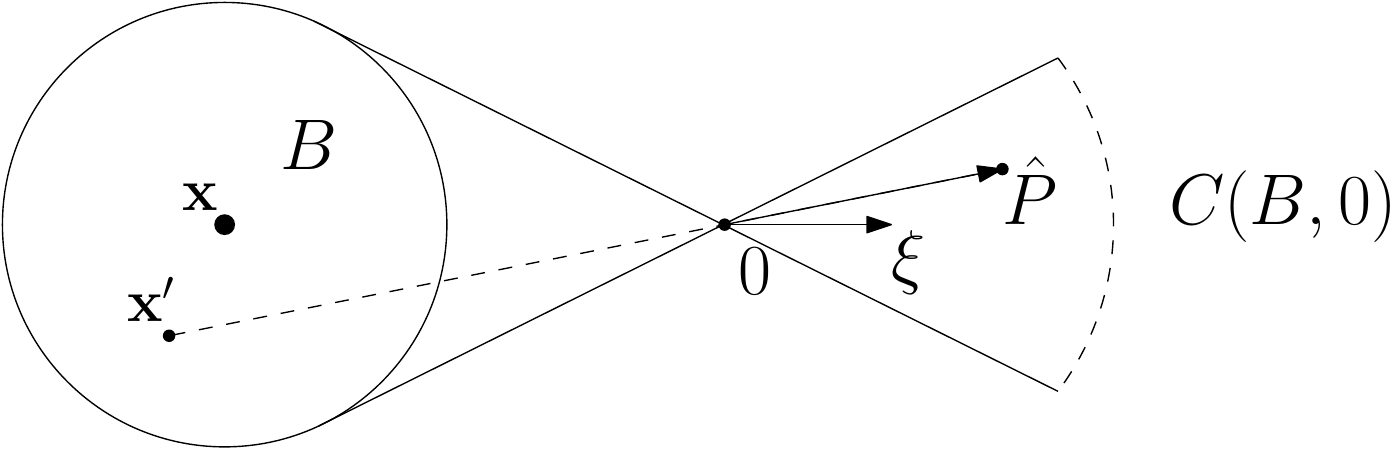}
\caption{A cone generated by $B_r(\mathbf{x})$ through $0$. There is a scalar $1 > c > 0$ and a unit vector $\xi$ through the centre axis of the cone for which the dot product of any $\hP \in C^+(B,0)$ with $\xi$ is at least $c \dnorm{\hP}$. }
\label{fig:Cone_dot_product}
\end{figure}

\begin{proposition} \label{propB2}
Suppose $0 \not\in B_r(\mathbf{x})$. There exists a $c \in (0,1)$ such that  $C^+(B_r(\mathbf{x}),0) = \{ \hP \in \R^3 : -\hP \cdot \mathbf{x} > c \dnorm{\hP} \dnorm{\mathbf{x}} \}$, see Fig \ref{fig:Cone_dot_product}.
\end{proposition}
\begin{proof}
If $\hP \in C^+(B_r(\mathbf{x}),0)$ then $\hP \neq 0$ because by definition $C^+(B_r(\mathbf{x}),0)$ is open and does not contain its apex. Hence the statement that $\hP \in C^+(B_r(\mathbf{x}),0)$ is equivalent to stating the existence of a line segment passing from $\hP$ through $0$ which intersects $B_r(\mathbf{x})$. This is equivalent to stating that $\norm{\frac{\hP}{\dnorm{\hP}} \cdot \mathbf{x}}^2 - \dnorm{\mathbf{x}}^2 + r^2 > 0$, 
and so $\frac{\norm{\hP \cdot \mathbf{x}}^2}{\dnorm{P}^2 \dnorm{\mathbf{x}}^2} > 1 - \frac{r^2}{\dnorm{\mathbf{x}}^2}$. Letting $c^2 = 1 - \frac{r^2}{\dnorm{\mathbf{x}}^2}$, noting that $r < \dnorm{\mathbf{x}}$ and observing that by definition $ - \hP \cdot \mathbf{x} > 0$, the conclusion follows.
\end{proof}

\begin{proposition} \label{prop:cone_diameter}
Let $C$ be a right-angled cone with base radius $b$ and height $h$. Let $x$ be the apex of the cone. Then $C \subset B_\rho(x)$ where $\rho = 2 \sqrt{b^2 + h^2}$.
\end{proposition}
\begin{proof} [Sketch of proof]
This follows from taking a planar cut of the cone containing its central axis, resulting in an isosceles triangle, and representing points in this triangle as a convex sum of the corners. 
\end{proof}
%

\subsection{Error Function} \label{subsec:error_function}

In the following, we prove that for each lenslet $\ell$, the local error function $\eps(\hat{\Delta},\ell)$ defined by \eqref{eq:photometric_error} has a unique minimum at $\hat{\Delta} = \Delta$, where $\Delta$ is the true distance of the scene in direction $\eta(\ell)$ and the first argument of $\eps$ is restricted to $(\Delta_\text{min},\infty)$. 

\begin{lemma} \label{lem:error_minima}
Let $\eps$ be the error function defined by \eqref{eq:photometric_error}. Let $\Delta$ be the true distance of the scene $\Omega$ in direction $\eta(\ell)$. 
Then $\eps(\Delta,\ell) = 0$ and if $\Delta_\text{min} < \hat{\Delta}_1 < \hat{\Delta}_2 < \Delta$ or $\Delta_\text{min} < \Delta < \hat{\Delta}_2 < \hat{\Delta}_1$, we have that $\eps(\hat{\Delta}_1,\ell) > \eps(\hat{\Delta}_2,\ell) > 0$. 
\end{lemma}
\begin{proof} 
Denote $P=\Delta\eta(\ell)$ and $Q=\iota(P)$, see \eqref{eq:Image_Map}.
Firstly, if $\hat{\Delta} = \Delta$, then transforming the integral in \eqref{eq:photometric_error} through the inverse projection map transforms the window $W(\hat{\Delta},\ell)$ to a single point $\hat{P} = P$ on the scene, and so the error is $0$ in this case. 

Let $Z$ denote the focus lense, which is a disc of radius $A$ (where $A$ is the aperture) normal to $\nu$. If $\hat{Q}$ is an image point estimate, and $\z \in Z$, then let $\pi_{\hat{Q}}(\z)$ denote the perspective projection of the point $\z$ through $\hat{Q}$ onto the pupilar plane $\Lambda$. 

Note that $\pi_{\hat{Q}}(\z) = \z + \frac{D}{\hat{Q}^z} \left( \z - \hat{Q} \right)$, where $\hat{Q}^z = \hat{Q} \cdot \nu$. Therefore, $\norm{\det D \pi_{\hat{Q}}(\z)}$ only depends on $\hat{Q}$ and is given by $\norm{\det D \pi_{\hat{Q}}(\z)} = \left( 1 + \frac{D}{\hat{Q}^z} \right)^2.$ 

 Now, consider $\eps(\hat{\Delta}_1, \ell) - \eps(\hat{\Delta}_2,\ell)$, and note that in either case we have that $\norm{\hat{\Delta}_1 - \Delta} > \norm{\hat{\Delta}_2 - \Delta}$. Then we have that
\begin{align*}
& \eps(\hat{\Delta}_1, \ell) - \eps(\hat{\Delta}_2,\ell) \\
= &\ \int_{W(\hat{\Delta}_1,\ell)} e(\ell',\hat{\Delta}_1,\ell) \left( 1 + \frac{D}{\hat{Q}_1^z} \right)^{-2} d \ell' \\
&\ - \int_{W(\hat{\Delta}_2,\ell)} e(\ell',\hat{\Delta}_2,\ell) \left( 1 + \frac{D}{\hat{Q}_2^z} \right)^{-2} d \ell' \\ 
= &\ \ \ \int_{Z} \dnorm{\b(P) - \b(\pi^{-1}_{\hP_1}(\z))}^2 d \zeta \\
 & - \int_{Z} \dnorm{\b(P) - \b(\pi^{-1}_{\hP_2}(\z))}^2 d \zeta \\
 > &\ 0.
\end{align*}

Here we have used Assumptions \ref{assume:increasing_colour_diff} and \ref{assume:convexity} and the fact that if the scene is convex then the further a point estimate $\hP$ is from the scene, the further the projection of a point on the focus lens through $\hP$ will be from the true point $P$. 
\end{proof}

\subsection{Point Trajectories} \label{subsec:Point_Trajectories}

The first observation is that if $\hP\in\Omega$ is a point on the scene then $v_\t(\hP)=0$ for all $t$ by Lemma \ref{lem:error_minima}.
This means that $\hP_t=\hP$ for all $t$ is a trajectory of \eqref{eq:Observer_Estimate}, and hence $\hP_t=\hP$ for \emph{some} $t$ implies $\hP_t=\hP$ for \emph{all} $t$ because solutions of \eqref{eq:Observer_Estimate} are assumed to be unique. 

The following result additionally states that if the point estimate lies in $\Sigma$ for some time $t$, it stays in $\Sigma$ for all future times, and if it lies in $\Sigma^c=\R^3\setminus\text{cl}(\Sigma)$ it stays there. 

\begin{proposition} \label{propC2b}
If $\hP_t \in \Omega$ then $\hP_\t \in \Omega$ for all $\t$.
If $\hP_t \in \Sigma$ then $\hP_\t \in \Sigma$ for all $\t\geq t$.
If $\hP_t \in \Sigma^c$ then $\hP_\t \in \Sigma^c$ for all $\t\geq t$.
\end{proposition}
\begin{proof}
We have already shown the first statement at the beginning of Section \ref{subsec:Point_Trajectories}.
Assume $\hP_t \in \Sigma$ and assume for a contradiction that $\hP_\t\not\in\Sigma$ for some $\t>t$. Because $\hP$ as defined by \eqref{eq:Observer_Estimate} is continuous, there exists an $s \in [t,\t]$ such that $\hP_s\in \Omega = \p \Sigma$. By the first statement it follows that $\hP_{s'}\in\Omega$ for all $s'$, a contradiction to $\hP_t\in\Sigma$.
The case $\hP_t \in \Sigma^c$ follows from a similar argument. 
\end{proof}

The goal of the remainder of this section is to establish that if a point estimate $\hP$ with initial condition $\hP_0 \in \Sigma$ has a limit point $Q$, then that limit point cannot be in $C^+(\mathbf{B},\hP_0)\cap\Sigma$. A similar statement holds for the case where $\hP_0 \in \Sigma^c$ with the obvious modifications to all the intermediate statements and proofs.

In subsection \ref{subsubsec:point_estimates}, we investigate general properties which must be true of any solution of \eqref{eq:Observer_Estimate} with $\hP_0\in\Sigma$. 
In subsection \ref{subsubsec:accumulation_points_are_limits} we show that every accumulation point of the trajectory $\hP$ is a limit point. 
In \ref{subsubsec:limit_contradiction} we establish that the assumption that the limit point of the trajectory $\hP$ is in $C^+(\mathbf{B},\hP_0)\cap\Sigma$ results in a contradiction.   

\subsubsection{Properties of Point Estimates} \label{subsubsec:point_estimates}

We begin by investigating the time set $T(\hP_0)$ from Definition \ref{def:TP0}.
\begin{proposition}\label{prop:open}
$T(\hP_0)$ is open. 
\end{proposition}
\begin{proof}
Let $t>0$ and express the point estimate $\hP_t$ in frame $\mathbf{C}$ as ${}^\mathbf{C} \hP_t$. Let ${}^\mathbf{C} \pi_0$ be the perspective projection of points in front of the camera through the optical centre expressed in frame $\mathbf{C}$ (in which it has constant coordinates $0$) onto the pupilar plane ${}^\mathbf{C} \Lambda$ which is constant in the frame $\mathbf{C}$, as is ${}^\mathbf{C} \Lambda^*$. 
Then, ${}^\mathbf{C} \hP_t = X^{-1}_t \hP_t$, which is continuous with respect to $t$ since $\hP_t$ and $X_t$ are, the latter by Assumption~\ref{assume:camera_contained}. Since ${}^{\mathbf{C}} \pi_0$ is continuous, ${}^{\mathbf{C}} \pi_0({}^\mathbf{C} \hP_t)$ is continuous with respect to $t$, and if ${}^{\mathbf{C}} \pi_0({}^\mathbf{C} \hP_t) \in {}^\mathbf{C} \Lambda^*$, there is a time interval $(a_t,b_t)$ containing $t$ such that ${}^{\mathbf{C}} \pi_0 ({}^\mathbf{C} \hP_\t) \in {}^\mathbf{C} \Lambda^*$ for all $\t \in (a_t,b_t)$. Now, $T(\hP_0) = \bigcup_{t \in T(\hP_0)} (a_t,b_t)$ which is open.
\end{proof}

The following proposition shows that for $t\in T(\hP_0)$ the vector field in \eqref{eq:Observer_Estimate} points into the interior of a cone with apex $\hP_t$ spanned by the ball $\mathbf{B}$ from Assumption \ref{assume:camera_contained}.
\begin{proposition} \label{propC0}
Let $t \in T(\hP_0)$ and $\hP_t \in \Sigma$ and $\hP_t \in \mathbf{B}$, where $\mathbf{B}$ is from Assumption \ref{assume:camera_contained}. Then $\hP_t + \dhP_t \in C^+(\mathbf{B},\hP_t)$. 
\end{proposition}
\begin{proof}
Let $\ell_t = \pi_{\mathbf{x}_t}(\hP_t)$ then $\dhP_t = - \nabla_1 \eps(\hP_t \cdot \eta_t(\ell_t),\ell_t) \eta_t(\ell_t)$ and $\nabla_1 \eps(\hP_t \cdot \eta_t(\ell_t),\ell_t) < 0$ by Lemma \ref{lem:error_minima}. Therefore, $\dhP_t$ is a positive multiple of $\eta_t(\ell_t)$ in this case and $\hP_t + \eta_t(\ell_t) \in C^+(\mathbf{B},\hP_t)$ implies $\hP_t + h \dhP_t \in C^+(\mathbf{B},\hP_t)$ for all $h > 0$ as $C^+(\mathbf{B},\hP_t)$ is a cone. 
\end{proof}

The following proposition gives the existence of some time interval $(t,t+\epsilon)$ for which the trajectory of a point estimate then remains within the cone $C^+(\mathbf{B},\hP_t)$ for all times within the time interval $(t,t+\epsilon)$. This is important for establishing the existence of a limit point for the trajectory. 

\begin{proposition} \label{propC1}
Let $t \in T(\hP_0)$ and $\hP_t\in\Sigma$ and $\hP_t \not\in \mathbf{B}$, where $\mathbf{B}$ is from Assumption \ref{assume:camera_contained}. Then there exists an $\eps>0$ such that $\hP_{t+h}\in C^+(\mathbf{B},\hP_t)\cap\Sigma$ for all $0<h<\eps$.
\end{proposition}
\begin{proof}
By Prop. \ref{propC0}, $\hP_t + \dhP_t \in C^+(\mathbf{B},\hP_t)$. Since $C^+(B,\hP_t)$ is open, there exists a $\d > 0$ such that $B_\d(\hP_t + \dhP_t) \subset C^+(\mathbf{B},\hP_t)$. But then $B_{\delta h}(\hP_t + h \dhP_t) \subset C^+(\mathbf{B},\hP_t)$ for all $h > 0$ since $C^+(\mathbf{B},\hP_t)$ is a cone. 

As $t$ is in $T(\hP_0)$ and $T(\hP_0)$ is open by Prop. \ref{prop:open}, we have $\dhP_t = \lim_{h \ra 0} \frac{\hP_{t+h} - \hP_t}{h}$. Hence there exists an $\eps > 0$ such that for all $0 < h < \eps$, we have $\dnorm{\hP_{t+h} - (\hP_t + h \dhP_t)} < \delta h$. It follows that $\hP_{t+h} \in C^+(\mathbf{B},\hP_t)$ and by Prop. \ref{propC2b} also $\hP_{t+h}\in\Sigma$ for all $0<h<\eps$. 
\end{proof}

The following proposition uses the previous proposition to produce a stronger result: that for every time $t \in T(\hP_0)$ and every time $\t > t$, the point estimate $\hP_{\t}$ is contained in the cone $C^+(\mathbf{B},\hP_t)$. 

\begin{proposition} \label{prop:subcone_containment}
Let $t \in T(\hP_0)$ and $\hP_t \in \Sigma$ and $\hP_t \not\in \mathbf{B}$ where $\mathbf{B}$ is from Assumption \ref{assume:camera_contained}. Then $\hP_{\t} \in C^+(\mathbf{B}, \hP_t)\cap\Sigma$ for all $\t > t$.
\end{proposition}
\begin{proof} 
Assume, to arrive at a contradiction, that there exists $\t>t$ with $\hP_{\t} \not\in C^+(\mathbf{B},\hP_t)\cap\Sigma$. 
By Proposition \ref{propC1}, $\hP_{t+h}\in C^+(\mathbf{B},\hP_t)\cap\Sigma$ for $h>0$ sufficiently small. 
Since $\hP$ is continuous in $t$, there is a smallest time $b\in (t,\t)$ such that $\hP_b \in \p (C^+(\mathbf{B},\hP_t)\cap\Sigma)$ and $\hP_s \in C^+(\mathbf{B},\hP_t)\cap\Sigma$ for all $s\in(t,b)$.
By Prop. \ref{propC2b}, $\hP_b\in\Sigma$ and hence $\hP_b\in \p (C^+(\mathbf{B},\hP_t)\cap\Sigma) \cap \Sigma=\p C^+(\mathbf{B},\hP_t)\cap\Sigma$. In particular, $\hP_b\not\in C^+(\mathbf{B},\hP_t)$.

If $b\in T(\hP_0)$ then there exists a nonempty open interval $(a,b)\subset(t,b)$ 
such that $(a,b)\subset T(\hP_0)$ as $T(\hP_0)$ is open by Prop. \ref{prop:open}.
If $b\not\in T(\hP_0)$ then $s\not\in T(\hP_0)$ and therefore $\dhP_s=0$ for all $s\in[b',b]$, where $b'=\sup\{s\in T(\hP_0)\,|\,s<b\}$, and there exists a nonempty open interval $(a,b')\subset (t,b')$ such that $(a,b')\subset T(\hP_0)$.
But then $\hP_{b'}=\hP_{b}\in \p (C^+(\mathbf{B},\hP_t)\cap\Sigma)$ and $b'=b$ as $b$ was minimal. It follows that there exists a nonempty open interval $(a,b)\subset(t,b)$ 
such that $(a,b)\subset T(\hP_0)$ also in this case.

In both cases we then have that there exists a nonempty open interval $(a,b)\subset T(\hP_0)$ such that $\hP_{s} \in C^+(\mathbf{B},\hP_t)\cap\Sigma$ for all $s \in(a,b)$.
By Prop \ref{propC0}, it follows that 
$\hP_{s}+\dhP_{s} \in C^+(\mathbf{B},\hP_{s})$ for all $s \in(a,b)$, and by Prop. \ref{prop:cones_generated_from_elements_are_subcones}, $\hP_t+\dhP_s \in C^+(\mathbf{B},\hP_t)$ for all $s \in(a,b)$. Recall that $\hP_b\not\in C^+(\mathbf{B},\hP_t)$.

For the remainder of the argument we change coordinates such that $\hP_t=0$. This is so we can apply Proposition \ref{propB2}. In the new coordinates $\dnorm{\mathbf{x}} > r > 0$, where $\mathbf{x}$ is the centre of the ball $\mathbf{B}$ of radius $r$, by our assumption that $\hP_t\not\in \mathbf{B}$. We now have $\hP_a \in C^+(\mathbf{B},0)$ and $\dhP_s \in C^+(\mathbf{B},0)$ for all $s \in(a,b)$ but $\hP_b \not\in C^+(\mathbf{B},0)$. Because $\hP$ is absolutely continuous on the interval $[a,b]$ we have:
\begin{align*}
- \hP_b \cdot \xi = &\ - \hP_a \cdot \xi + \int_a^b - \dhP_s \cdot \xi \ ds \\ 
> &\ c \dnorm{\xi} \dnorm{\hP_a} + \int_a^b c \dnorm{\xi} \dnorm{\dhP_s} \ ds \\ 
\geq &\ c \dnorm{\xi} \dnorm{\hP_a} + c \dnorm{\xi} \dnorm{\int_a^b \dhP_s \ ds} \\
= &\ c \dnorm{\xi} \dnorm{\hP_a} + c \dnorm{\xi} \dnorm{\hP_b - \hP_a} \\
\geq &\ c \dnorm{\xi} \dnorm{\hP_b} \\
\end{align*}
which implies $\hP_b \in C^+(\mathbf{B},0)$ by Proposition \ref{propB2} (note the $>$ sign on the second line). This is a contradiction to 
$\hP_b \not\in C^+(\mathbf{B},0)$ and it follows that $\hP_{\tau} \in C^+(\mathbf{B}, \hP_t)\cap\Sigma$ for all $\tau > t$.
\end{proof}

The following two results are the main results of this subsection. 

\begin{proposition} \label{prop:nested_cones}
Let $\hP_t \in \Sigma$ and $\hP_t \not\in \text{cl}(\mathbf{B})$ where $\mathbf{B}$ is from Assumption \ref{assume:camera_contained}. Then $\hP_\t\in C_0^+(\mathbf{B},\hat{P}_t)\cap\Sigma$ and $C^+(\mathbf{B},\hP_\t) \subset C^+(\mathbf{B},\hP_t)$ for all $\t\geq t$. \end{proposition}
\begin{proof}
Clearly $\hP_\t\in C_0^+(\mathbf{B},\hat{P}_t)$ implies $C^+(\mathbf{B},\hP_\t) \subset C^+(\mathbf{B},\hP_t)$ by Prop. \ref{prop:cones_generated_from_elements_are_subcones}, and $\hP_\t\in\Sigma$ for all $\t\geq t$ by Prop. \ref{propC2b}. Hence we only need to prove $\hP_\t\in C_0^+(\mathbf{B},\hat{P}_t)$ for all $\t\geq t$. The case $\tau=t$ is immediate, so let $\tau>t$ for the remainder of the proof.
Let $t\in T(\hP_0)$ then the statement follows from Prop. \ref{prop:subcone_containment}.
Let $t\not\in T(\hP_0)$ then $s\not\in T(\hP_0)$ and therefore $\dhP_s=0$ for all $s\in[t,t']$, where $t'=\inf\{s\in T(\hP_0)\,|\,s>t\}$.
Note that $t'$ is finite by Assumption \ref{assume:persistence}. It follows that
$\hP_s=\hP_t$ for all $s\in[t,t']$ and there exists a nonempty open interval 
$(t',b)\subset T(\hP_0)$. The case $\tau\leq t'$ is now immediate, so assume $\tau>t'$ for the remainder of the proof. 

Recall $\hP_{t'}=\hP_t\not\in\text{cl}(\mathbf{B})$. Since $\hP$ is continuous, there exists $b'\in(t',b)$ such that $\hP_s\not\in \text{cl}(\mathbf{B})$ for all
$s\in(t',b')$. Now, we have two cases: either $\t \in (t',b')$ or $t \not \in (t',b')$. 

Assume $\t \in (t',b')$ for now, recall that $\hP_t'=\hP_t$ and assume for a contradiction that $\hP_\t \not\in C^+_0(B,\hP_{t'})$. 
Then $\hP_{t'} \not\in C^-(\mathbf{B},\hP_\t)$ by Prop. \ref{prop:cone_reversal}.
Furthermore, $\hP_\t \in C^+(\mathbf{B},\hP_s)$ for all $s \in (t',\t)$ by Prop. \ref{prop:subcone_containment}, and hence $\hP_s \in C^-(\mathbf{B},\hP_\t)$ for all $s \in (t',\t)$ by Prop. \ref{prop:cone_reversal}.
Since $\hP_s$ is inside the open cone $C^-(\mathbf{B},\hP_\t)$ and $\hP_{t'} \not\in C^-(\mathbf{B},\hP_\t)$ and $\hP_{t'}\not\in\text{cl}(\mathbf{B})$, it follows that there exists $\delta>0$ such that $\dnorm{\hP'-\hP_{t'}}\geq\delta$ for all $\hP'\in C^-(\mathbf{B},\hP_s)$.

Repeating the argument, by Prop \ref{prop:subcone_containment}, $\hP_s \in C^+(\mathbf{B},\hP_{s'})$ for all $s' \in (t',s)$, and hence $\hP_{s'} \in C^-(\mathbf{B},\hP_{s})$ for all $s' \in (t',s)$ by Prop. \ref{prop:cone_reversal}.
This implies $\dnorm{\hP_{s'}-\hP_{t'}}\geq\delta$ for all $s' \in (t',s)$ and hence $\lim_{s' \ra t'} \dnorm{\hP_{s'} - \hP_{t'}} \geq \d$, which contradicts continuity of $\hP$ at $t'$.
Therefore, if $\t \in (t',b')$ then $\hP_\t \in C^+_0(\mathbf{B},\hP_{t'})=C^+_0(\mathbf{B},\hP_{t})$.

For the second case, if $\t \not \in (t',b')$, then we take any $\t' \in (t',b')$, and conclude using the previous argument that $\hP_{\t'} \in C_0^+(\mathbf{B},\hP_t)$. Using Prop. \ref{prop:subcone_containment} we have that $\hP_s \in C^+(\mathbf{B},\hP_{\t'})$ for all $s > \t'$, and this latter set is contained in $C^+(\mathbf{B},\hP_t)$ by Prop. \ref{prop:cones_generated_from_elements_are_subcones}, and therefore $\hP_\t \in C^+_0(\mathbf{B},\hP_t)$ also in this case. 
\end{proof}

\begin{proposition} \label{prop:subcone_existence}
Let $\hP_t \in \Sigma$ and $\hP_t \not\in \mathbf{B}$ where $\mathbf{B}$ is from Assumption \ref{assume:camera_contained}. Then there exists a $t^+ > t$ such that $\hP_\t \in C^+(\mathbf{B},\hP_t)\cap\Sigma$ for all $\t > t^+$. 
\end{proposition}
\begin{proof}
By Assumption \ref{assume:persistence} there exists a $t^+ > t$ such that
$t^+\in T(\hP_0)$. By Prop. \ref{prop:nested_cones}, $\hP_{t^+}\in C_0^+(\mathbf{B},\hat{P}_t)\cap\Sigma$ and hence $\hP_{t^+}\in \Sigma$ and $\hP_{t^+}\not\in\mathbf{B}$. By Prop. \ref{prop:subcone_containment}, $\hP_\t \in C^+(\mathbf{B},\hP_{t^+})\cap\Sigma$ for all $\t > t^+$. By Prop. \ref{prop:nested_cones}, $C^+(\mathbf{B},\hat{P}_{t^+}) \subset C^+(\mathbf{B},\hP_t)$ and the result follows. 
\end{proof}

\begin{figure}[ht!]
\centering
\includegraphics[width=3.0in]{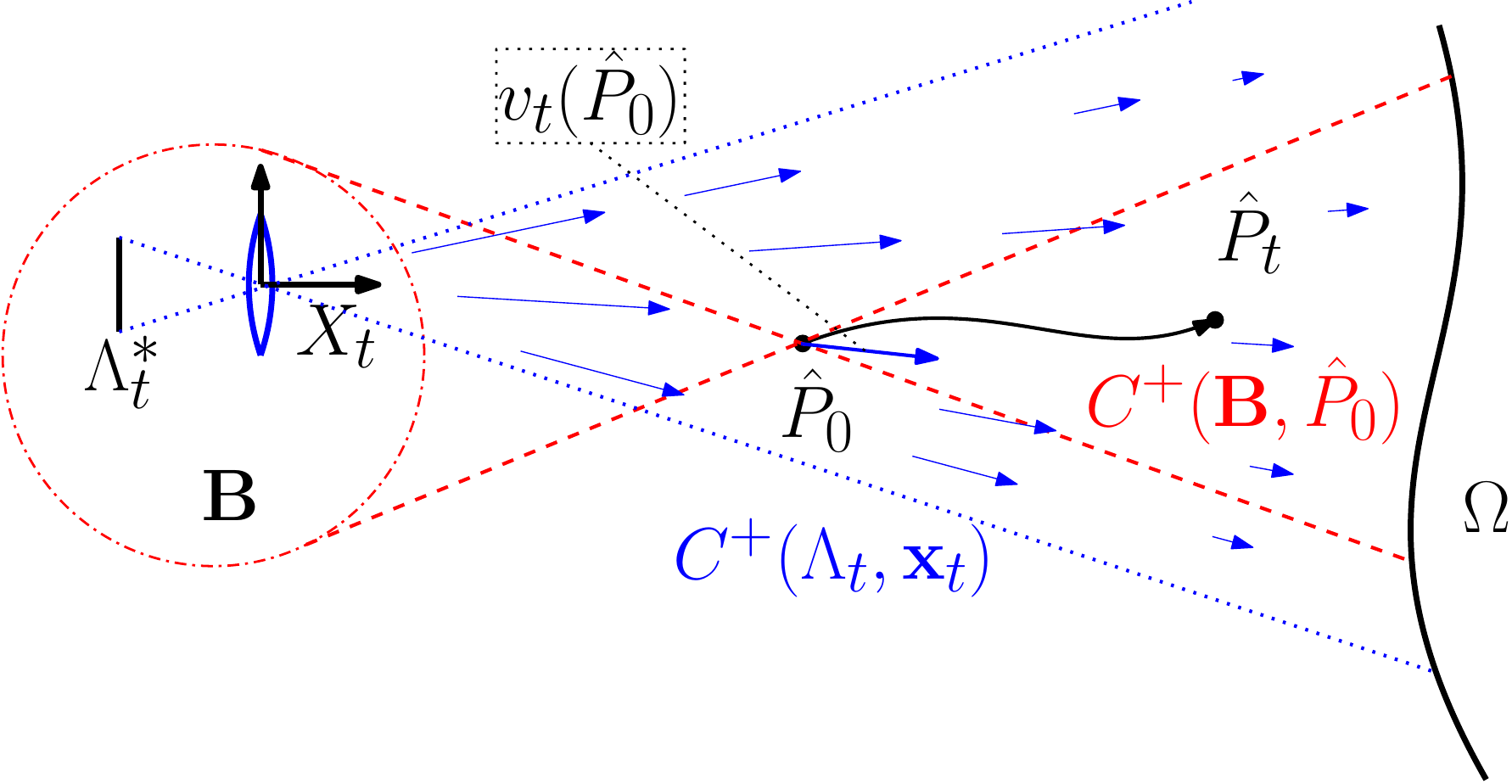}
\caption{A initial point estimate $\hat{P}_0 \in \Sigma$, $\hP_0 \not \in \mathbf{B}$ has its trajectory $\hP$ contained in the pointed cone $C_0^+(\mathbf{B},\hP_0)$. The observer produces a vector field $v_t$ which always points away from the optical centre of the camera. The set of points for which the vector field can be non-zero is the cone $C^+(\Lambda^*_t,\mathbf{x}_t)$, where $\mathbf{x}_t$ is the optical centre. }
\label{fig:Trajectory}
\end{figure}

We have now established that if $\hP_0 \in \Sigma$ and $\hP_0 \not \in \mathbf{B}$ then $\hP_t \in C^+_0(\mathbf{B},\hP_0)\cap\Sigma$ for all $t \geq 0$, see Fig. \ref{fig:Trajectory}. Since the trajectory $\hP$ is contained in a bounded set, it is a simple consequence of the Bolanzo-Weierstrass theorem that the trajectory has an accumulation point in the closure of that set. 

\subsubsection{Accumulation points are limit points} \label{subsubsec:accumulation_points_are_limits}

The following two propositions establish that any accumulation point of the trajectory $\hP$ must be a limit point. 

\begin{proposition} \label{propC3}
Let $\hP_0 \in \Sigma$ and $\hP_0 \not\in \mathbf{B}$ where $\mathbf{B}$ is from Assumption \ref{assume:camera_contained}.
If $Q$ is an accumulation point of the trajectory $\hP$ then $Q \in C^+(\mathbf{B},\hP_t)$ for all $t \geq 0$. 
\end{proposition}
\begin{proof}
Suppose for a contradiction that there were a $t \geq 0$ such that $Q \not\in C^+(\mathbf{B},\hP_t)$. 
By Prop \ref{prop:subcone_existence} there exists a $t^+ > t$ such that $\hP_\t \in C^+(\mathbf{B},\hP_t)$ for all $\t > t^+$. By Prop.  \ref{prop:cones_generated_from_elements_are_subcones}, $\text{cl}(C^+(\mathbf{B},\hP_\t)) \subset C^+(\mathbf{B},\hP_t)$ for all $\t > t^+$ and since the latter set is open, $Q \not\in C^+(\mathbf{B},\hP_t)$ has a strictly positive distance from all the former sets. In particular, there exists a $\d > 0$ such that for all $\t > t^+$, we have $\dnorm{Q - \hP_{\t}} > \d$, which contradicts the assumption that $Q$ is an accumulation point. 
\end{proof}

\begin{figure}[ht!]
\centering
\includegraphics[width=3.0in]{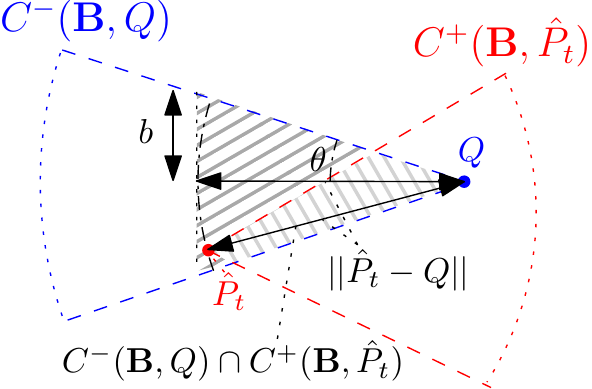}
\caption{The cones $C^-(\mathbf{B},Q)$ and $C^+(\mathbf{B},\hP_t)$ and their intersection are illustrated. In the darker grey shaded region is a right-angled cone containing the intersection with base radius $b$ and height $\dnorm{\hat{P}_t-Q}$. }
\label{fig:Right_cone_containment}
\end{figure}

\begin{proposition} \label{propC4}
Let $\hP_0 \in \Sigma$ and $\hP_0 \not\in \text{cl}(\mathbf{B})$ where $\mathbf{B}$ is from Assumption \ref{assume:camera_contained}.
Any accumulation point $Q$ of the trajectory $\hP$ is a limit point. 
\end{proposition}
\begin{proof}
Fix $t\geq 0$. By Prop. \ref{prop:nested_cones}, $\hP_{t}\in C_0^+(\mathbf{B},\hat{P}_0)\cap\Sigma$ and hence $\hP_{t}\in \Sigma$ and $\hP_{t}\not\in\text{cl}(\mathbf{B})$.
Again by Prop. \ref{prop:nested_cones}, $\hP_{\t}\in C_0^+(\mathbf{B},\hat{P}_t)$ for all $\t>t$.
By Prop. \ref{propC3}, $Q \in C^+(\mathbf{B},\hP_\t)$ which by Prop. \ref{prop:cone_reversal} implies $\hP_\t\in C^-(\mathbf{B},Q)$, for all $\t\geq t$. 
Therefore, $\hP_t\in C^-(\mathbf{B},Q)$ and $Q\in C^+(\mathbf{B},\hP_t)$, see Fig. \ref{fig:Right_cone_containment}, and $\hP_\t \in C^-(\mathbf{B},Q) \cap C_0^+(\mathbf{B},\hP_t)$ for all $\t>t$.
 
Let $\th$ be the opening angle of the cone $C^-(\mathbf{B},Q)$. The set $C^-(\mathbf{B},Q) \cap C_0^+(\mathbf{B},\hP_t)$ is contained in a right-angled cone of base radius $b=\dnorm{\hP_t - Q} \tan \th$ and height $\dnorm{\hP_t-Q}$ because $\norm{\frac{Q}{\dnorm{Q}} \cdot (\hP_t - Q)} \leq \dnorm{\hP_t - Q}$, see Fig. \ref{fig:Right_cone_containment} and recall that $\mathbf{B}$ is centred at $0$. 

By Prop. \ref{prop:cone_diameter}, this right-cone is contained in an open ball around $Q$ of radius $2 \sqrt{1 + \tan^2 \th} \dnorm{\hP_t-Q}$, and therefore $\dnorm{\hP_t - Q} < \eps$ implies $\dnorm{\hP_\t - Q} < 2 \sqrt{1 + \tan^2 \th} \cdot \eps$ for all $\t>t$. 

This implies that $Q$ is a limit point, because given $\rho>0$ there exists a $t\geq 0$ such that $\dnorm{\hP_t-Q} < \rho/(2 \sqrt{1 + \tan^2 \th})$ since $Q$ is an accumulation point, and hence $\dnorm{\hP_\t - Q} < \rho$ for all $\t>t$. 
\end{proof}

\subsubsection{The limit point can not be in $C^+(\mathbf{B},\hP_0)\cap\Sigma$} \label{subsubsec:limit_contradiction}

\begin{figure}[ht!]
\centering
\includegraphics[width=3.0in]{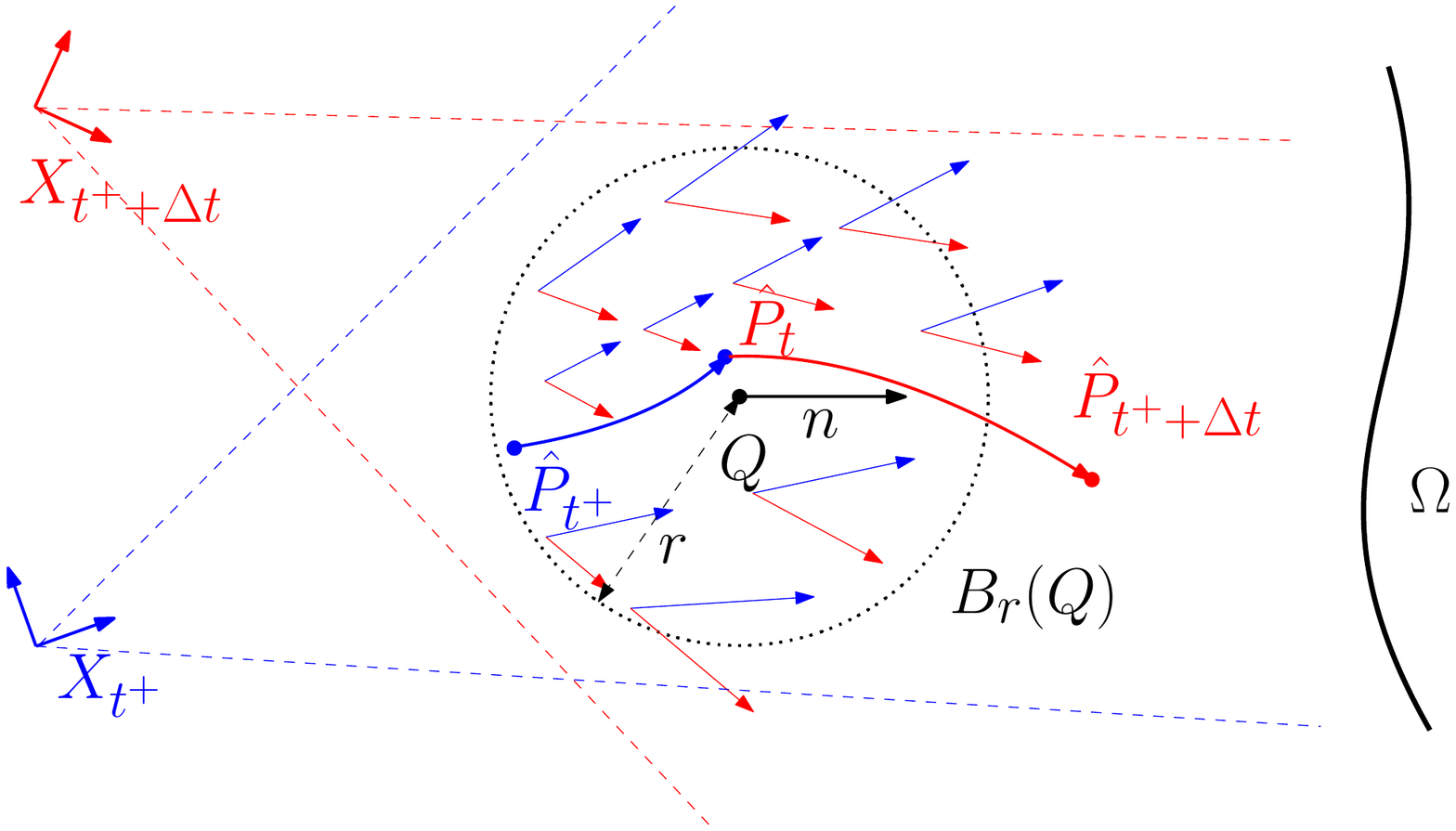}
\caption{There is an open ball of radius $r>0$ around a limit point $Q\in\Sigma$ for which point estimates entering the ball eventually leave. In this diagram, the vector field $v$ is shown for two different times shown in red and blue. There is a vector $n$ and a $c>0$ for which each of the vectors $v_\t(\hP')$ assigned to a point $\hP'$ in the ball at a time $\t \in [t^+,t^+ + \Delta t]$ satisfies $n \cdot v_\t(\hP') \geq c$.}
\label{fig:Limit_Points}
\end{figure}

The following proposition implies that if the trajectory $\hP$ enters a certain nonempty open ball around the limit point $Q$ it will eventually leave that ball, see Fig. \ref{fig:Limit_Points}.

\begin{proposition} \label{propD3}
Let $\hP_0 \in \Sigma$ and $\hP_0 \not\in \mathbf{B}$ where $\mathbf{B}$ is from Assumption \ref{assume:camera_contained}.
Let $Q \in \Sigma$ be a limit point of the trajectory $\hP$, and let $\Delta t$ be the length of time from Assumption \ref{assume:persistence}. 
Then there exists a direction $n$, a $c>0$, an $r>0$, and a sequence $(t^+_i)_{i=1}^\infty$ of times with $t^+_i > 0$ for all $i\in\N$ and $\lim_{i\to\infty}t^+_i=\infty$, such that for all $i\in\N$ and for all times $\t \in [t^+_i,t^+_i + \Delta t]$ and all points $\hP' \in B_r(Q)$, $n \cdot v_\t(\hP') \geq c$. 
\end{proposition}
\begin{proof} 
Let $\rho>0$ be the radius from Assumption \ref{assume:persistence} and choose $0 <  r < \frac{\rho}{2}$ such that $B_{r}(Q)\subset\Sigma$ and $B_r(Q)\cap\mathbf{B}=\emptyset$. Such an $r$ exists since $\Sigma$ is open and $Q$ has a positive distance from $\mathbf{B}$ by Prop. \ref{propC3}. 
Because $\mathbf{B}$ and $B_r(Q)$ are both convex and non-intersecting, there exists a separating hyperplane $\Gamma$ between them. Let $n$ be the unit normal vector to this hyperplane pointing in the direction of $Q$.

Since $Q$ is a limit point, there exists a time $t\geq 0$ such that $\hP_\t \in B_{r}(Q)$ for all $\t > t$, and by Assumption \ref{assume:persistence}, there exists a sequence $(t^+_i)_{i=1}^\infty$ with $t^+_i > t \geq 0$ for all $i\in\N$ and $\lim_{i\to\infty}t^+_i=\infty$, such that $\pi_{\mathbf{x}_\t}(B_{\rho}(\hP_\t)) \subset \Lambda^*_\t$ for all $i\in\N$ and $\t \in [t^+_i, t^+_i + \Delta t]$. 
Because $r < \frac{\rho}{2}$, this implies $\pi_{\mathbf{x}_\t}(\text{cl}(B_{r}(Q))) \subset \Lambda^*_\t$ for all $i\in\N$ and $\t \in [t^+_i, t^+_i + \Delta t]$.

Now fix $\hP'\in \text{cl}(B_{r}(Q))$, $i\in\N$ and $\t \in [t^+_i, t^+_i + \Delta t]$ and let $\ell_\t = \pi_{\mathbf{x}_\t}(\hP')$. Then $\ell_\t \in \Lambda^*_\t$ and hence $v_\t(\hP') = 
-\nabla_1 \eps (\hP' \cdot \eta_\t(\ell_\t),\ell_\t)\eta_\t(\ell_\t)$.  
Because $\eta_\t(\ell_\t)$ points from $\ell_\t\in\mathbf{B}$ into the direction of $\hP'\in \text{cl}(B_{r}(Q))$ on the other side of the hyperplane $\Gamma$, and because $\nabla_1 \eps (\hP' \cdot \eta_\t(\ell_\t),\ell_\t)<0$ by Lemma \ref{lem:error_minima}, it follows that $n \cdot v_\t(\hP') > 0$.

Changing coordinates to the main lens $Z$ as in the proof of Lemma \ref{lem:error_minima} gives
\begin{equation} \label{eq:gradient_expanded}
\nabla_1 \eps(\hat{\Delta},\ell) = \int_{Z} D_{\hat{\Delta}} \dnorm{\beta(P) - \beta(\pi^{-1}_{\hat{\Delta} \eta}(\zeta))}^2 d \zeta, 
\end{equation}
where $\eta=\eta(\ell)$ and $\pi^{-1}_{\hat{\Delta} \eta} (\zeta)$ is the perspective projection from $Z$ through $\hP=\hat{\Delta} \eta(\ell)$ to $\Omega$. Note that the integrand is defined as the derivative of the sum of absolute differences of a composition of perspective projections and the smooth brightness map $\b$, and so the expression on the right hand side of \eqref{eq:gradient_expanded} is a continuous function $F$ of $\hat{\Delta}$ and 
$\eta$, as long as $\eta$ points away from the main lens and towards $P$.
It follows that $-F(\hat{\Delta},\eta)\eta\cdot n$ attains its minimum $c>0$ on the compact set 
$\{(\hat{\Delta},\eta)\,|\,\hat{\Delta}\eta\in\text{cl}(B_{r}(Q)) \text{ and }
\hP_0+\eta\in\text{cl}(C^+(\mathbf{B},\hP_0))\}$. Here we have used that $n$ points towards $Q$ and $Q\in C^+(\mathbf{B},\hP_0)$ by Prop. \ref{propC3}.

Since $\hP'\in \text{cl}(B_{r}(Q))$ and $\hP_0+\eta_\t(\ell_\t)\in\text{cl}(C^+(\mathbf{B},\hP_0))$ by Prop. \ref{prop:nested_cones},
it follows that $n \cdot v_\t(\hP') = -\nabla_1 \eps (\hP' \cdot \eta_\t(\ell_\t),\ell_\t)\eta_\t(\ell_\t) \cdot n \geq c$.
\end{proof}

It now follows that there can not be a limit point of the trajectory $\hP$ in $C^+(\mathbf{B},\hP_0)\cap\Sigma$.

\begin{lemma} \label{lemD1}
Let $\hP_0 \in \Sigma$ and $\hP_0 \not\in \mathbf{B}$ where $\mathbf{B}$ is from Assumption \ref{assume:camera_contained}. 
Then the trajectory $\hP$ has no limit point in the set $C^+(\mathbf{B},\hP_0) \cap \Sigma$.
\end{lemma}
\begin{proof}
Suppose for a contradiction that the point $Q\in C^+(\mathbf{B},\hP_0) \cap \Sigma$ were a limit point of the trajectory $\hP$. 

Let $\Delta t$ be the length of time from Assumption \ref{assume:persistence}. By Prop. \ref{propD3} there exists a direction $n$, a $c>0$, an $r>0$, and a sequence $(t^+_i)_{i=1}^\infty$ of times with $t^+_i > 0$ for all $i\in\N$ and $\lim_{i\to\infty}t^+_i=\infty$, such that for all $i\in\N$ and for all times $\t \in [t^+_i,t^+_i + \Delta t]$ and all points $\hP' \in B_r(Q)$, $n \cdot v_\t(\hP') \geq c$.

Pick $r' < \min\{r, \frac{c \cdot \Delta t}{2}\}$ then there exists a time $t\geq 0$ such that $\hP_\t \in B_{r'}(Q)$ for all $\t > t$ because $Q$ is a limit point. 
Pick $i\in\N$ with $t^+_i > t$ then $\hP_{t^+_i+\Delta t} \not\in B_{r'}(Q)$ because $n \cdot v_\t(\hP') \geq c$ for all $\t \in [t^+_i,t^+_i + \Delta t]$ and all $\hP' \in B_{r'}(Q)\subset B_r(Q)$, a contradiction.
\end{proof}

\subsection{Limit Points Must Lie on the Scene}

The main result is now restated. 

\begin{theorem}
Let $\hP_0\in\Sigma$ and $\hP_0\not\in\text{cl}(\mathbf{B})$, where $\mathbf{B}$ is from Assumption \ref{assume:camera_contained}. Then there exists a point $P \in \Omega$ such that $\lim_{t \ra \infty} \hP_t = P$.
\end{theorem}
\begin{proof}
By Prop. \ref{prop:nested_cones}, $\hP_t$ is contained within $C_0^+(\mathbf{B},\hP_0) \cap \Sigma$ for all $t\geq 0$. The Bolanzo-Weierstrass theorem implies that the trajectory $\hP$ has an accumulation point $Q$ within the closure of that set. By Prop. \ref{propC4}, $Q$ is a limit point. By Lemma \ref{lemD1}, $Q\not\in C^+(\mathbf{B},\hP_0) \cap \Sigma$ but by Prop. \ref{propC3}, $Q\in C^+(\mathbf{B},\hP_0)$.  Therefore, $\hP_t$ has a limit on $\p \Sigma = \Omega$. 
\end{proof}

The case $\hP_0\in\Sigma^c$ follows along the same lines, replacing positive cones with negative cones where appropriate. The case $\hP_0\in\Omega$ follows trivially from Prop. \ref{propC2b}.

\section*{Acknowledgment}

This research was supported by the Australian Research Council through the ARC Discovery Project DP160100783 ``Sensing a complex world: Infinite dimensional observer theory for robots.''

\bibliography{IEEEabrv,./bibliography}

\end{document}